\documentclass[10pt]{article} \topmargin=0pt 
\usepackage[hang]{caption2}
\usepackage{float}
\usepackage{epsfig}
\usepackage{amssymb,amsthm}
\usepackage{amsmath}
\usepackage{graphicx, color, subfigure,tikz}
\usepackage{epstopdf}
\usepackage{multirow}
\usepackage{fancyhdr}
\usepackage[english]{babel}
\usepackage{fancybox}
\usepackage{bm}
\usepackage{verbatim}

\usepackage{colortbl}
\usepackage{mathrsfs}
\usepackage{siunitx}
 \usepackage{booktabs}
\usepackage{hhline}
\newcommand{\be}{\begin{equation}}
\newcommand{\ee}{\end{equation}}

\newcommand\QQ[1]{\textcolor{blue}{#1}}

\newtheorem{theorem}{Theorem}[section]
\newtheorem{lemma}{Lemma}[section]
\newtheorem{corollary}{Corollary}[section]
\newtheorem{definition}{Definition}[section]

\newtheorem{remark}{Remark}[section]
\newtheorem{exa}{Example}[section]
\graphicspath{{Figures/}}

\begin{document}
\baselineskip 1pc
\renewcommand{\refname}{\begin{flushleft}{{\bf\small
Reference}}\end{flushleft}}
\title {{\bf Stability analysis of the Eulerian-Lagrangian finite volume methods for nonlinear hyperbolic equations in one space dimension}
 \thanks{\small The first author was supported by Simons Foundarion 961585. The second and last author were supported by NSF grant NSF-DMS-2111253, Air Force Office of Scientific Research FA9550-22-1-0390 and Department of Energy DE-SC0023164.}
}
\author{Yang Yang\thanks{Department of Mathematical Sciences, Michigan Technological University, Houghton, MI 49931, USA. E-mail:yyang7@mtu.edu},
\quad Jiajie Chen\thanks{Department of Mathematics, University of Delaware, Newark, DE 19716, USA. E-mail: jiajie@udel.edu },
\quad Jing-Mei Qiu\thanks{Department of Mathematics, University of Delaware, Newark, DE 19716, USA. E-mail: jingqiu@udel.edu }
}
\date{}
\maketitle
\begin{center}
\begin{minipage}{13.4cm}
{\bf \small Abstract:} In this paper, we construct a novel Eulerian-Lagrangian finite volume (ELFV) method for nonlinear scalar hyperbolic equations in one space dimension. It is well known that the exact solutions to such problems may contain shocks though the initial conditions are smooth, and direct numerical methods may suffer from restricted time step sizes. To relieve the restriction, we propose an ELFV method, where the space-time domain was separated by the partition lines originated from the cell interfaces whose slopes are obtained following the Rakine-Hugoniot junmp condition. Unfortunately, to avoid the intersection of the partition lines, the time step sizes are still limited. To fix this gap, we detect effective troubled cells (ETCs) and carefully design the influence region of each ETC, within which the partitioned space-time regions are merged together to form a new one. Then with the new partition of the space-time domain, we theoretically prove that the proposed first-order scheme with Euler forward time discretization is total-variation-diminishing and maximum-principle-preserving with {at least twice} larger time step constraints than the classical first order Eulerian method for Burgers' equation. Numerical experiments verify the optimality of the designed time step sizes.
{\small

}
{\small \bf AMS(2000) Subject Classifications: } {\small 65M08, 65M12}
\end{minipage}
\end{center}
\section{Introduction}
\label{sec1}
Nonlinear hyperbolic equations have many significant applications in traffic flow, large scale supply-chains, compressible gas dynamics, magneto-hydrodynamics and continuum physics, etc. It is well known that the exact solutions for nonlinear hyperbolic equations may develop shocks though the initial condition is sufficiently smooth, and direct high-order numerical methods may yield oscillatory approximations near the shocks. Therefore, numerical methods for nonlinear hyperbolic conservation laws are challenging to design.

Finite volume (FV) methods are of great interesting in solving nonlinear hyperbolic equations \cite{leveque2002finite}. The methods apply numerical fluxes to guarantee the local mass conservation. Traditional FV methods are coupled with strong-stability-preserving (SSP) Runge-Kutta (RK) time discretization, and they may suffer from the limited time step sizes \cite{gottlieb2009high}. There are several previous attempts in the literature to relieve the restriction. For example, in \cite{qiu2008convergence} a third-order Godunov-type scheme with total variation bounded property is proposed for conservation laws with large time stepping sizes. However, the scheme in \cite{qiu2008convergence} relies on an extra tracking of characteristics via the Lax-Hopf formula.  Another attempt was to design the Eulerian-Lagrangian (EL) FV \cite{healy1998solution, huang2012eulerian, Joseph2022} and the EL discontinuous Galerkin (DG) methods \cite{Cai2021}. Yet, many of these existing research are designed for linear or nonlinear transport dominated models without shocks, i.e. crossing of characteristics.  
In the EL framework, often times the Rankine-Hugoniot (R-H) jump condition was used to approximate the speed of the characteristic lines, 
giving a partition of the space-time domain between the current and next time levels. It is easy to show that if the characteristics do not intersect, the numerical approximations from the first-order spatial discretization and Euler forward time integration is total-variation-diminishing (TVD) and maximum-principle-preserving (MPP) (see Lemma \ref{modify_Harten}). Unfortunately, in the presence of shocks, to avoid the intersection of the characteristic lines, the time step sizes are still restricted. There is little previous works in resolving such an issue and theoretically discussing the precise upper bound of the time step sizes for EL methods with the presence of shocks. 

In this paper, we develop a novel EL FV method using a uniform mesh for 1D nonlinear hyperbolic problems with the consideration of shocks. A brief outline of the algorithm is as follows:
\begin{enumerate}
\item Partition space-time regions by approximating characteristic lines with the R-H jump condition from cell interfaces. We name the characteristic lines as partition lines in the rest of the paper.
\item Detect effective troubled cells (ETCs), on which the partition lines originated from cell boundaries intersect within a pre-determined time step size or the difference between the numerical approximations in the ETC and its neighbor cell is larger than some threshold.
\item Construct the influence region of each ETC (see Definition \ref{ir}). 
\item Merge the cells within the influence region of each ETC to form a larger one. In general, the partition lines based on the new partition will not intersect under the original time step size.
\item Keep the original numerical fluxes at the boundaries of the merged cell.
\item Update numerical solutions by using the EL FV method \cite{Joseph2022} on the new mesh.
\item Map the numerical solutions to the background uniform distributed cells.
\end{enumerate}
There are two main challenging aspects that we need to address in designing effective EL FV methods.
\begin{itemize}
\item The size of the influence region is not easy to determine. Intuitively, the larger the influence region, the larger the time step size are allowed to avoid the intersection of partition lines. On the other hand, large influence regions may yield poor resolutions near the ETCs that often involve shocks.
{In Section \ref{secsimple}, via manipulated simple examples, we demonstrate that if the influence region is larger than some threshed, further enlargement of the influence region will not result in the increase in time step sizes if the influence regions of the ETCs do not intersect.} 
The influence region we proposed mainly contains five cells centered at the ETC. In some special cases, we may need to include one more cell in the influence region to gain the TVD property or reduce one cell in the influence region to gain better resolution. See Definition \ref{ir} for the detailed construction of the influence region for each ETC.
\item Determine the precise upper bound for the time step size with TVD and MPP for the first-order scheme. {On one hand, with merged cells in the influence region, the original piece-wise constant numerical approximation is being $L^2$ projected to a new constant; and the R-H jump condition will yield new slopes for the partition lines for the merged region. On the other hand, we intend to keep the partition lines and the numerical fluxes at boundaries of the influence region for better resolution.}
With the new average value in the merged cell, but the original partition lines and numerical fluxes, the classical TVD analysis cannot be applied. To fix this gap, i.e. to enable the use of classical TVD analysis, we reassign the numerical solutions on cells of the influence region according to the following principles: 
\begin{enumerate}
    \item [(a)] The reassigned numerical solutions on the influence region conserves the ``total mass" within the influence region;
    \item [(b)] The total variation of reassigned numerical solutions is no larger than the original one to guarantee the TVD property; 
    \item [(c)] The reassigned numerical solutions are MPP; 
    \item [(d)] The partition lines determined from the R-H condition originated from the interior cell interfaces of the influence region, as well as the two boundaries of the influence region based on the updated numerical approximations, do not intersect within the pre-determined time step size.
\end{enumerate}
\end{itemize}

The above considerations guide the design of influence region with provable stability properties. In particular, to keep the original numerical flux at the two boundaries of the influence region, reassigned solutions on the two boundary cells of the influence region have to stay the same. Moreover, since we need to keep the ``total mass", the influence region must contain more than three cells.
Due to symmetry, most the influence regions we construct contain five cells. However, in some extreme cases, we may need to merge one more cell next to the boundary or in some unlikely mild cases, we can merge one less cell (See Definition \ref{ir}). 


With suitable construction of the influence region, it is possible to apply the idea for the traditional EL FV method to update the solution and theoretically find the optimal time step size for the TVD and MPP properties. The theoretical analysis is achieved by designing suitable admissible set based on the properties (a)-(d) given above, and find optimal reassigned values of solutions on the interior of the influence region to satisfy the last criteria.  Note that the above strategy to reassign numerical solutions is only for the sake of numerical analysis to obtain the upper bound on the time step constraint (see Theorem 3.1). In the numerical implementation, one would simply first perform the merging of influence regions of ETCs; followed by a first order EL FV method on the updated mesh.  

 

The rest of the paper is organized as follows. We develop the EL FV methods in Section \ref{sec2}. The stability analysis without and with the merging strategy will be presented in Section \ref{sec3}. Numerical experiments will be given in Section \ref{secexample} to verify the theoretical analysis. We will end with concluding remarks in Section \ref{secend}.

\section{The Eulerian-Lagrangian finite volume formulation}
\label{sec2}\setcounter{equation}{0}

In this section, we propose the FV spatial discretization with Euler forward time discretization under the EL framework for scalar nonlinear hyperbolic equations in one space dimension with merging strategy, extending the work proposed in \cite{Cai2021}. The precise definition of the terminologies used in constructing the numerical method will be summarized in the end of this section.

We consider the following nonlinear hyperbolic equation
\begin{equation}
\left\{\begin{array}{l}u_t+f(u)_x = 0.\\u(x,0)=u_0\end{array}\right.
\label{scalar1d}
\end{equation}
subject to periodic boundary conditions and assume $b\leq u_0\leq a$. We give a partition of the computational domain $\Omega=[0, 2\pi]$ as
$$0=x_{\frac12}< x_{\frac32}<\cdots< x_{N+\frac12} =2\pi,$$
and denote $I_j=[x_{j-\frac12}, x_{j+\frac12} ]$ as the cells with length $\Delta x_j=x_{j+\frac12}-x_{j-\frac12}$. In this paper, we consider uniform mesh and denote $\Delta x=\Delta x_j$, $j=1,\cdots,N$. Let $t^n$ be the $n-$th time level and denote $\Delta t^n =t^{n+1}-t^n$ as the time step size. In this paper, we consider uniform partition of the time domain, and use $\Delta t$ for the time step size. However, this assumption is not essential. Finally, we denote $\tilde\lambda=\frac{\Delta t}{\Delta x}$.

The EL FV scheme starts from constructing a partition of the space-time domain $\Omega_j$'s (see figure \ref{schematic_1d}), followed by a forward Euler update of the solution on a non-uniform mesh $I_j^*$'s, and an $L^2$ re-projection to obtain the solution on the original uniform mesh $I_j$'s. The specific procedure is outlined below.  
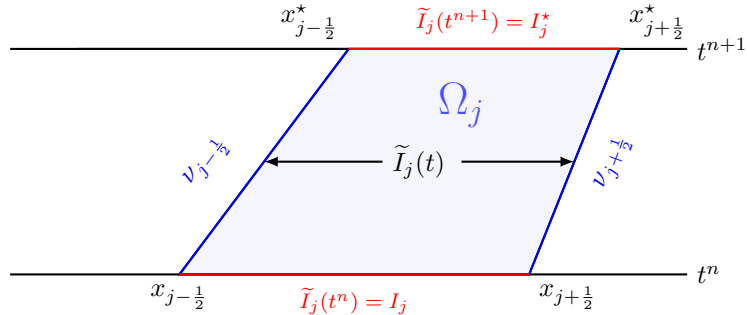
\begin{figure}[h!]
\centering
\begin{tikzpicture}[x=1.5cm,y=1.5cm]
  \begin{scope}[thick]

  \draw[fill=blue!4] (0.,2) -- (-1.5,0) -- (1.6,0) -- (2.4,2)
      -- cycle;
   \node[blue!70, rotate=0] (a) at ( 1. ,1.5) {\LARGE $\Omega_j$ };
    \draw[black]                   (-3,0) node[left] {} -- (3,0)
                                        node[right]{$t^{n}$};
    \draw[black] (-3,2) node[left] {} -- (3,2)
                                        node[right]{$t^{n+1}$};
     \draw[] (-2.4,2) -- (0,2) node[above left] {$x_{j-\frac12}^\star$};
     \draw[] (0,2) -- (2.4,2) node[above right] {$x_{j+\frac12}^\star$};


            \draw[blue,thick] (0.,2) node[left] {} -- (-1.5,0)
                                        node[black,below]{$x_{j-\frac12}$ };

                \draw[blue,thick] (2.4,2) node[left] {} -- (1.6,0)
                                        node[black,below right]{ $x_{j+\frac12}$};
\draw [color=red,xshift=0pt,yshift=0pt](-1.5,0) -- (1.6,0) node [red,midway,xshift=0cm,yshift=-10pt]
{\footnotesize $\widetilde{I}_j(t^{n}) = I_{j}$};

\draw [color=red,xshift=0pt,yshift=0pt]
(0,2) -- (2.4,2) node [red,midway,xshift=0cm,yshift=10pt]
{\footnotesize $\widetilde{I}_j(t^{n+1}) = I_{j}^\star$};

\draw[-latex,thick](0.3,1)node[right,scale=1.]{$\widetilde{I}_j(t)$}
        to[out=180,in=0] (-0.75,1) node[above left=2pt] {};

\draw[-latex,thick](1,1)node[right,scale=1.]{ }
        to[out=0,in=180] (2,1) node[above left=2pt] {};

  \node[blue, rotate=54] (a) at (-1.25,1.) { $\nu_{j-\frac12}$ };
  \node[blue, rotate=70] (a) at (2.35,1.) { $\nu_{j+\frac12}$ };
  \end{scope}
\end{tikzpicture}
\caption{Illustration for the space-time region for the EL FV formulation.}\label{schematic_1d}
\end{figure}

\begin{enumerate}
\item Given the numerical approximations $u$ at $t^n$. We use $\nu_{j+\frac12}$ to be an approximation of the local characteristics speed of the original problem \eqref{scalar1d} at $x_{j+\frac12}$. In this paper, we take
\begin{equation}\label{nu}
\nu_{j+\frac12} =
\begin{cases}
\frac{[f]}{[u]}|_{{j+\frac12}}, & \text{if} \ u_{j+\frac12}^+  \neq u_{j+\frac12}^-, \\
f'\left(u_{j+\frac12}\right), &  \text{otherwise},
\end{cases}
\end{equation}
where $[u]=u^+-u^-$ is the jump of $u$, with $u^+$ and $u^-$ being the right and left limit of the reconstructed numerical approximation $u$ at the cell interface.
\item Draw partition lines originated from $x_{j\pm\frac12}$ with speed  $\nu_{j\pm\frac12}$. If the numerical cell average in $I_j$ is significantly larger than that in the right neighbor or significantly smaller than that in the left neighbor, or the partition lines intersect before $t=t^{n+1}$, then $I_j$ is defined as a troubled cell. In the latter case, we compute the time that the intersection appears, i.e. to find $t^\star_{j}$ such that 
$$
x_{j-\frac12}+\nu_{j-\frac12}t^\star=x_{j+\frac12}+\nu_{j+\frac12}t^\star.
$$
Otherwise, we define $I_j^\star=[x^\star_{j-\frac12},x^\star_{j+\frac12}]$ as the line segment between the two partition lines at $t=t^{n+1}$ with length $\Delta x^\star_j=x^\star_{j+\frac12}-x^\star_{j-\frac12}$, and define $t^\star_j=\Delta t$. It is easy to see
\begin{equation}\label{cell_relation}
\Delta x_j^\star=\Delta x+\nu_{j+\frac12}\Delta t-\nu_{j-\frac12}\Delta t.
\end{equation}
The region enclosed by the two partition lines, $I_j$ and $I_j^\star$ is called the space-time domain $\Omega_j$. Moreover, we define $\tilde{I}_j(t)$ as the line segment between the two partition lines at time $t\in[t^n,t^{n+1}]$, and the left and right endpoints are $\tilde{x}_{j-\frac12}(t)$ and $\tilde{x}_{j+\frac12}(t)$, respectively.
\item For each troubled cell $I_i$, we construct the influence region of $I_i$. In general, we choose the influence region to be 
\begin{equation}
\label{eq: m}
I_{i-m}, \cdots,  I_{i}, \cdots, I_{i+m}
\end{equation}
with some $m$. In some special circumstances, the influence region may contain one more cell or one less cell. The precise construction of the influence region will be given in Definition \ref{ir}. Then we merge the cells in the influence region and denote the new cell as $I_i$. If there is another troubled cell between $I_{i-1}$ and $I_{i+1}$, we can select $I_i$ as the effective troubled cell (ETC) and the influence region is designed based on $I_i$ only. If an influence region does not overlap with that of other troubled cells, then we say the corresponding troubled cell is isolated. Otherwise, we combine the overlapping influence regions together.  In Figure~\ref{fig: merge_spacetime} we illustrate the merging of five cells in the influence region of an isolated ETC (region bounded in red lines).
\begin{figure}[htb]
    \centering
    \includegraphics[width=3.0in]{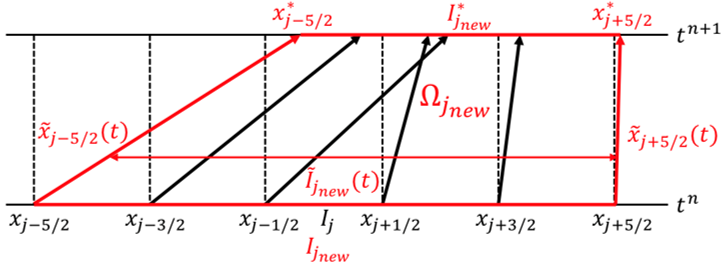}
    \caption{The general case (case I4 in Def. \ref{ir}) of merging of influence region for an isolated troubled cell.}
    \label{fig: merge_spacetime}
\end{figure}

\item After merging the cells, we can construct the first order EL FV scheme on the updated space time domain $\Omega_j$ by integrating \eqref{scalar1d} over $\Omega_j$ to obtain
\begin{align}
0 =& \int_{t^n}^{t^{n+1} } \int_{\widetilde{I}_j(t)} u_t\ dxdt
 + \int_{t^n}^{t^{n+1} } \int_{\widetilde{I}_j(t)}  f(u)_x\ dxdt \nonumber \\
 =& \int_{t^n}^{t^{n+1} } \left[
\frac{d}{dt} \int_{\widetilde{I}_j(t) }u\ dx -\left( \left.\nu u\right|_{\widetilde{x}_{j+\frac12} } - \left.\nu u\right|_{\widetilde{x}_{j-\frac12} }\right) +\left( \left.f\right|_{\widetilde{x}_{j+\frac12} }  - \left. f \right|_{\widetilde{x}_{j-\frac12} }  \right)
\right]\ dt.
\label{eq: int}
\end{align}
Considering the differential form of \eqref{eq: int}, we have
\begin{equation}
\frac{d}{dt} \int_{\widetilde{I}_j(t)}u\ dx + \left. \left(f-\nu u\right) \right|_{\widetilde{x}_{j+\frac12} } - \left. \left(f-\nu u\right) \right|_{\widetilde{x}_{j-\frac12} }=0.
\label{mol}
\end{equation}
Let
\begin{equation}
F_{j\pm\frac12}(u) \doteq f(u) - \nu_{j\pm\frac12} u,
\label{eq: F}
\end{equation}
we can rewrite \eqref{mol} as
\begin{equation}
\frac{d}{dt} \int_{\widetilde{I}_j(t)}u\ dx + \left.\hat{F} \right|_{\widetilde{x}_{j+\frac12}(t) } - \left. \hat{F} \right|_{\widetilde{x}_{j-\frac12}(t) }=0.
\label{mol_2}
\end{equation}
The proposed EL FV method with Euler forward time discretization is
$$
\Delta x_j^\star \bar{u}^{\star}_j-\Delta x \bar{u}_j + \Delta t\left(\hat{F}_{j+\frac12}  - \hat{F}_{j-\frac12}\right)=0.
$$
where $\bar{u}_j$ and $\bar{u}_j^{\star}$ are the numerical cell averages on $I_j$ at time level $n$ and $I_j^\star$ at time level $n+1$, respectively. $\hat{F}$ is the numerical flux given as
$$
\hat{F}_{j+\frac12}=\frac12\left(F_{j+\frac12}(u^-_{j+\frac12})+F_{j+\frac12}(u^+_{j+\frac12})-\alpha_{j+\frac12}(u^+_{j+\frac12}-u^-_{j+\frac12})\right).
$$
Due to \eqref{nu} we have
$$
F_{j+\frac12}(u_{j+\frac12}^-)\stackrel{\eqref{eq: F}}{=} f(u^-_{j+\frac12})-\nu_{j+\frac12} u^-_{j+\frac12}\stackrel{\eqref{nu}}{=}f(u^+_{j+\frac12})-\nu_{j+\frac12} u^+_{j+\frac12}\stackrel{\eqref{eq: F}}{=}F_{j+\frac12}(u^+_{j+\frac12}):=F_{j+\frac12},
$$
leading to
$$
\hat{F}_{j+\frac12}=F_{j+\frac12}-\frac12\alpha_{j+\frac12}(u^+_{j+\frac12}-u^-_{j+\frac12}).
$$
We can show that the numerical flux given in \eqref{flux} is consistent, by checking that 
$\hat{F}_{j+\frac12}(u,u)=F_{j+\frac12}(u)=f(u)-\nu_{j+\frac12}u.$
In this paper, we consider $f$ to be differentiable and strictly convex. For a monotone numerical flux, we choose
$$
\alpha_{j+\frac12}=\max\{f'(u^+_{j+\frac12})-\nu_{j+\frac12},\nu_{j+\frac12}-f'(u^-_{j+\frac12}),0\}.
$$

\item Finally, we map the solution back to the original background uniform mesh by an $L^2$ projection.
\end{enumerate}

If we consider the Burger's equation, i.e. $f(u)=\frac{u^2}2$, the EL FV scheme for the Burgers' equation is given as
\begin{equation}\label{mol_3_first}
\Delta x_j^\star {u}^{\star}_j-\Delta x {u}_j + \Delta t\left(\hat{F}_{j+\frac12}  - \hat{F}_{j-\frac12}\right)=0,
\end{equation}
where
\begin{equation}\label{flux}
\hat{F}_{j+\frac12}
=F_{j+\frac12}-\frac12\alpha_{j+\frac12}(u_{j+1}-u_{j}),
\end{equation}
with
\begin{equation}\label{F}
F_{j+\frac12}(u)=\frac{u^2}2-\nu_{j+\frac12}u,\quad F_{j+\frac12}=F_{j+\frac12}(u_{j})=F_{j+\frac12}(u_{j+1}),
\end{equation}
and
\begin{equation}\label{alpha}
\alpha_{j+\frac12}=\max\left\{\frac{u_{j+1}-u_j}2,0\right\},\quad \nu_{j+\frac12}=\frac{u_{j+1}+u_j}2.
\end{equation}
The choice of $\alpha$ in \eqref{alpha} is 0 if the local Riemann problem at the cell interface yields a shock wave, i.e. $u_{j+1}<u_j$. Since we use Rankine-Hugoniut jump condition to construct the partition line, no penalty is necessary in this case. However, for rarefaction waves, the traditional Lax-Friedrichs flux should be applied, and $\alpha$ is not 0.
In the rest of the paper, we consider first-order spatial discretization and Burgers' equation for simplicity, while the algorithm and theoretical results  
can be extended to general strictly convex flux function $f$.

\section{The stability analysis for Burger's equation}
\label{sec3}\setcounter{equation}{0}
In this section, we proceed to prove the stability of the proposed scheme for Burger's equation under suitable time step size. We consider the TVD stability and the maximum-principle of the numerical approximations. Suppose $u=\{u_j\}_{j=1}^N$ is the numerical approximation for the problems with compact support, then we define
$$
TV(u)=\sum_{j=1}^{N}|u_{j+1}-u_j|
$$
as the total variation of $u$. For periodic boundary conditions, the definition is similar.
A scheme is TVD if $TV(u^{n+1})\leq TV(u^{n})$. Moreover, suppose the initial condition $u_0$ satisfies $b\leq u^0\leq a$, then we say the numerical solutions are MPP if $b\leq u^n \leq a$ for all $n>0$. We first prove that if no troubled cells exist corresponding to the case of $m=0$ as in eq.~\eqref{eq: m} (a typical situation before a shock appears), the first-order scheme is TVD and MPP. Then we consider the general case with merge of influence region corresponding to $m=2$ (a typical situation with the presence of shocks). 

\subsection{Preliminary results for $m=0$}
\label{secm0}
In this subsection, we consider a special case that $m=0$, i.e. no cells merge, and try to prove some preliminary results. The most useful lemma to obtain the TVD stability is the following generalized Harten's Lemma extending the results given in \cite{Harten}.
\begin{lemma}\label{Harten_mod}
Suppose the solution to problem \eqref{scalar1d} has compact support or satisfies periodic boundary conditions. If the scheme can be written as
$$
u^{n+1}_j=u_j+\lambda_j\left(C_{j+\frac12}\Delta_+u_j-D_{j-\frac12}\Delta_-u_j\right)
$$
where
$$
\Delta_+u_j=u_{j+1}-u_j,\quad\Delta_-u_j=u_j-u_{j-1}.
$$
and
$C_{j+\frac12}\geq0$, $D_{j+\frac12}\geq0$, $\lambda_j\geq0$, $\lambda_jC_{j+\frac12}+\lambda_{j+1}D_{j+\frac12}\leq 1$, then the scheme is TVD.
\end{lemma}
\begin{proof}
We apply the $\Delta_+$ operator, then
{\setlength\arraycolsep{2pt}\begin{eqnarray*}
\Delta_+ u_j^{n+1}&=&\Delta_+ u_j+\lambda_{j+1}C_{j+\frac32}\Delta_+u_{j+1}-\lambda_jC_{j+\frac12}\Delta_+u_j-\lambda_{j+1}D_{j+\frac12}\Delta_-u_{j+1}+\lambda_jD_{j-\frac12}\Delta_-u_j\\
&=&\Delta_+ u_j+\lambda_{j+1}C_{j+\frac32}\Delta_+u_{j+1}-\lambda_jC_{j+\frac12}\Delta_+u_j-\lambda_{j+1}D_{j+\frac12}\Delta_+u_j+\lambda_jD_{j-\frac12}\Delta_+u_{j-1}\\
&=&(1-\lambda_jC_{j+\frac12}-\lambda_{j+1}D_{j+\frac12})\Delta_+u_j+\lambda_{j+1}C_{j+\frac32}\Delta_+u_{j+1}+\lambda_jD_{j-\frac12}\Delta_+u_{j-1}
\end{eqnarray*}}
Therefore,
{\setlength\arraycolsep{2pt}\begin{eqnarray*}
\sum_j|\Delta_+ u_j^{n+1}|&\leq& \sum_{j}(1-\lambda_jC_{j+\frac12}-\lambda_{j+1}D_{j+\frac12})|\Delta_+u_j|+\sum_j\lambda_{j+1}C_{j+\frac32}|\Delta_+u_{j+1}|+\sum_j\lambda_jD_{j-\frac12}|\Delta_+u_{j-1}|\\
&=&\sum_{j}(1-\lambda_jC_{j+\frac12}-\lambda_{j+1}D_{j+\frac12})|\Delta_+u_j|+\sum_j\lambda_jC_{j+\frac12}|\Delta_+u_j|+\sum_j\lambda_{j+1}D_{j+\frac12}|\Delta_+u_j|\\
&=&\sum_{j}|\Delta_+u_j|.
\end{eqnarray*}}
\end{proof}

\begin{lemma}\label{modify_Harten}
Suppose the solution to problem \eqref{scalar1d} has compact support. Assume 
$\{\Omega_j\}_{j=1}^N$ is the partition of the space-time domain $\Omega\times[t^n,t^{n+1}]$ with no trouble cells. If the initial condition $u_0$ satisfies $b\leq u_0\leq a$, then the first-order numerical approximation obtained from \eqref{mol_3_first} is TVD and MPP under the condition that
$\Delta t\leq\lambda\Delta x.$ Hence the EL FV scheme with a final step of $L^2$ projection to the uniform background cells is TVD and MPP.
\end{lemma}
\begin{proof}
We apply \eqref{cell_relation} to \eqref{mol_3_first} to obtain
$$
\Delta x_j^\star u^{\star}_j=(\Delta x_j^\star-\nu_{j+\frac12}\Delta t+\nu_{j-\frac12}\Delta t)u_j - \Delta t\left(\hat{F}_{j+\frac12}  - \hat{F}_{j-\frac12}\right),
$$
which further yields
$$
u^\star_j=u_j+\lambda_j\left[(\hat{F}_{j-\frac12}+\nu_{j-\frac12}u_j)-(\hat{F}_{j+\frac12}+\nu_{j+\frac12}u_j)\right].
$$
where $\lambda_j=\frac{\Delta t}{\Delta x^\star_j}.$ By using \eqref{flux}, we have
$$
u^\star_j=u_j+\lambda_j\left[\left(F_{j-\frac12}-\frac12\alpha_{j-\frac12}(u_{j}-u_{j-1})+\nu_{j-\frac12}u_j\right)-\left(F_{j+\frac12}-\frac12\alpha_{j+\frac12}(u_{j+1}-u_{j})+\nu_{j+\frac12}u_j\right)\right].
$$
Thanks to \eqref{F}, we obtain
{\setlength\arraycolsep{2pt}\begin{eqnarray*}
u^\star_j&=&u_j+\lambda_j\left[\left(f(u_j)-\frac12\alpha_{j-\frac12}(u_{j}-u_{j-1})\right)-\left(f(u_j)-\frac12\alpha_{j+\frac12}(u_{j+1}-u_{j})\right)\right]\\
&=&u_j+\frac{\lambda_j}2\left[\alpha_{j+\frac12}(u_{j+1}-u_{j})-\alpha_{j-\frac12}(u_{j}-u_{j-1})\right]\\
&=&\frac{\lambda_j}2\alpha_{j+\frac12}u_{j+1}+\left(1-\frac{\lambda_j}2\alpha_{j+\frac12}-\frac{\lambda_j}2\alpha_{j-\frac12}\right)u_j+\frac{\lambda_j}2\alpha_{j-\frac12}u_{j-1}.
\end{eqnarray*}}
By Lemma \ref{Harten_mod}, to obtain TVD and MPP numerical approximations, we need
$$
\frac{\lambda_j}2\alpha_{j+\frac12}+\frac{\lambda_{j+1}}2\alpha_{j+\frac12}\leq 1,\quad \frac{\lambda_j}2\alpha_{j+\frac12}+\frac{\lambda_j}2\alpha_{j-\frac12}\leq1,
$$
and the sufficient conditions are
\begin{equation}\label{con1}
\alpha_{j+\frac12}\lambda_j\leq1,\quad \alpha_{j-\frac12}\lambda_j\leq1.
\end{equation}
We consider $\alpha_{j+\frac12}\lambda_j\leq 1$ only, and the other condition can be obtained following the same lines. Moreover, we assume $\alpha_{j+\frac12}\neq0$, i.e. $\alpha_{j+\frac12}=\frac{u_{j+1}-u_j}2$, with $u_{j+1}>u_j$. It is easy to verify that
{\setlength\arraycolsep{2pt}\begin{eqnarray*}
\alpha_{j+\frac12}\lambda_j\leq 1&\Leftrightarrow&\frac{\Delta t}{\Delta x+\nu_{j+\frac12}\Delta t-\nu_{j-\frac12}\Delta t}\leq\frac2{u_{j+1}-u_j}\\
&\Leftrightarrow&\frac{u_{j+1}-u_j}2\Delta t\leq\Delta x+\frac{u_{j+1}+u_j}2\Delta t-\frac{u_j+u_{j-1}}2\Delta t\\
&\Leftrightarrow&\frac{u_{j-1}-u_j}2\Delta t\leq \Delta x
\end{eqnarray*}}
Since $I_j$ is not a troubled cell, then $u_{j-1}-u_j\leq\frac2\lambda.$ If $u_{j-1}-u_j\leq0$, then we have $\frac{u_{j-1}-u_j}2\Delta t\leq0\leq \Delta x$. If $u_{j-1}-u_j>0$, then $\frac{u_{j-1}-u_j}2\Delta t\leq\Delta x,$ 
Finally, it is easy to see that the $L^2$ projection to the background mesh in the final step preserves the maximum-principle and does not increase the total variation for first-order schemes. 
We complete the proof.
\end{proof}

\subsection{The main theorem with merging regions around trouble cells}\label{secsimple}
In the previous subsection, we discussed the case with $m=0$. In practice, a limited time step size is necessary, since the characteristics may quickly intersect near the shocks. To obtain larger time steps, we consider merging cells by first categorizing trouble cells, then identifying ETCs, followed by defining the corresponding influence regions.

\begin{definition}\label{def}
Suppose $u_{j-1},\ u_j,\ u_{j+1}$ are the numerical approximations on $I_{j-1},\ I_j, I_{j+1},$ respectively. $I_j$ is called a troubled cell of type I for $\lambda$ if
\begin{equation}\label{tc1}
u_{j-1}>u_{j+1}+\frac2\lambda.
\end{equation}
Here $\lambda$ is associated with an upper bound for time step constraint as in $\Delta t\le \lambda \Delta x$.
In addition, $I_j$ is called a troubled cell of type II for $\lambda$ if it is not a troubled cell of type I one and satisfies
\begin{equation}\label{tc2}
u_{j-1}>u_j+\frac2\lambda, \textrm{ and } u_{j-1}\geq u_{j+1}\geq u_j.
\end{equation}
and $I_j$ is called a troubled cell of type III for $\lambda$ if it is not a troubled cell of type I one and satisfies
\begin{equation}\label{tc3}
u_j>u_{j+1}+\frac2\lambda, \textrm{ and } u_j\geq u_{j-1}\geq u_{j+1},
\end{equation}
Moreover, $I_j$ is called a troubled cell of type IV for $\lambda$ if it is not a troubled cell of type I or type II and
\begin{equation}\label{tc4}
u_{j-1}>u_j+\frac2\lambda,
\end{equation}
and $I_j$ is called a troubled cell of type V for $\lambda$ if it is not a troubled cell of type I or type III and
\begin{equation}\label{tc5}
u_j>u_{j+1}+\frac2\lambda.
\end{equation}
\end{definition}

\begin{center}
\begin{figure}[htb]
    \centering
    \includegraphics[width=2.5in]{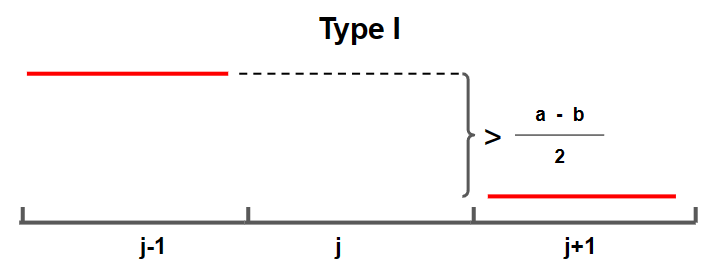} 
    \includegraphics[width=2.5in]{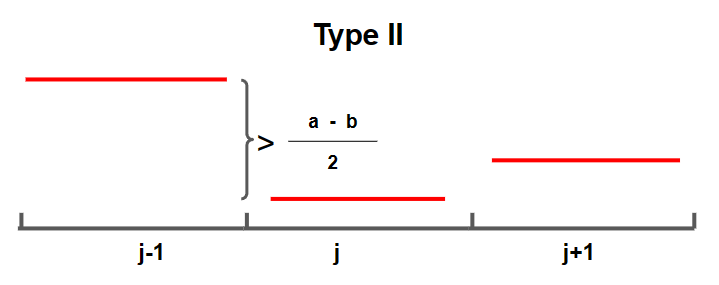}
    \includegraphics[width=2.2in]{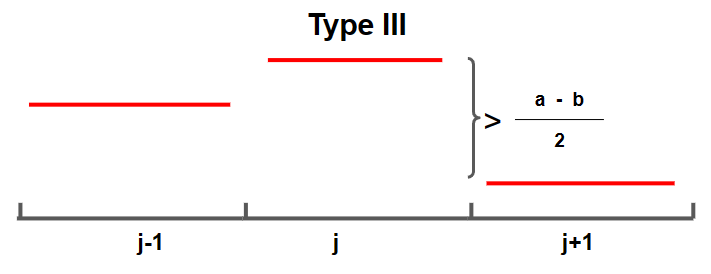} 
    \includegraphics[width=2.2in]{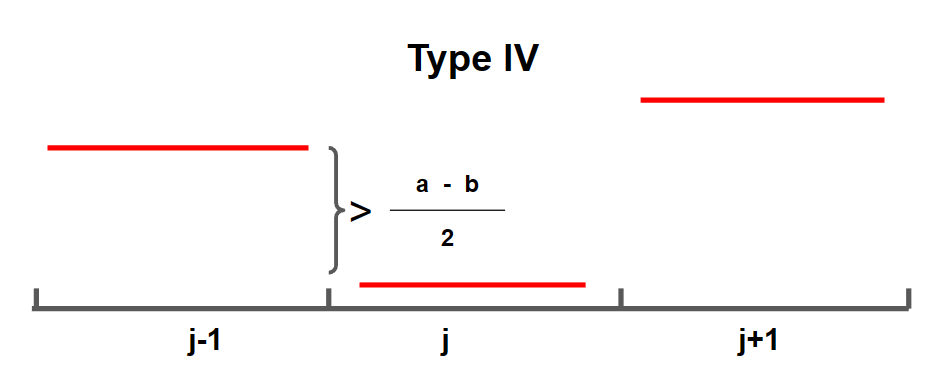}
    \includegraphics[width=2.2in]{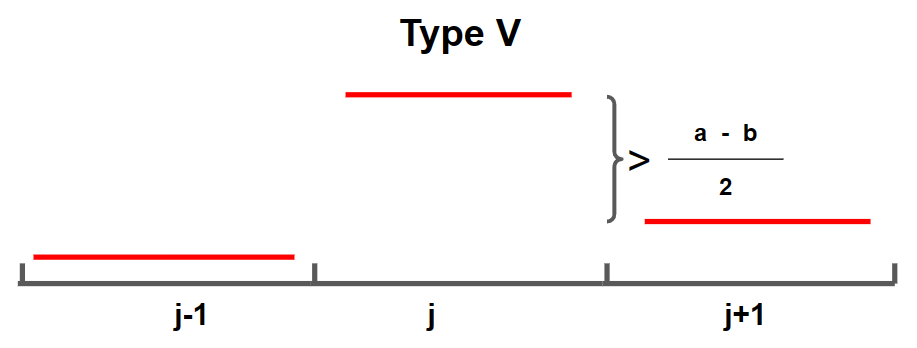}
    \caption{Diagrams of different types of trouble cells in Definition \ref{def}.}
    \label{fig:tctypes}
\end{figure}
\end{center}

In the definition of troubled cells given above, condition \eqref{tc1} is triggered if the partition lines originated from $x_{j\pm\frac12}$ intersect at $t< t^n+\lambda\Delta x$, while \eqref{tc2}, \eqref{tc3}, \eqref{tc4} and \eqref{tc5} are used for strong shock waves generated at the cell interfaces. Finally, \eqref{tc4} and \eqref{tc5}, though unlikely appear together, may be encountered in a rarefaction wave with strong numerical oscillations. It is easy to show that the above conditions to categorize troubled cells are mutually exclusive. Figure~\ref{fig:tctypes} illustrates various situations corresponding to different types of trouble cells. If two troubled cells are connected, we will design the influence region based on one of the troubled cells, and such a troubled cell is defined as an ETC. 
The ETCs are selected sequentially from left to right as follows: 
{\em 
We search the computational domain from left to right, if the first troubled cell, say $I_{j-1}$, is a troubled cell of type V, then $I_{j}$ is denoted as an ETC. Otherwise $I_{j-1}$ is defined as an ETC.
}
Once an ETC $I_j$ is identified, we define its influence region $R_j$ as below. 
\begin{definition}\label{ir}
Suppose $I_j$ is an ETC, and the numerical approximations on cell $I_i$, $i=j-3,\cdots,j+3$ are $s_\ell,$ $z_\ell$, $z_1$, $z_2$, $z_3$, $z_r$, $s_r$, respectively. Assume the initial condition is bounded by $a\geq u_0\geq b$, then the influence region is defined as follows
\begin{enumerate}
\item [I1:] \label{ir0} If $I_j$ is a troubled cell of type IV, then $I_{j-1}$ is a troubled cell of type V, and the influence region contains $I_i$, $i=j-2,\cdots,j+1$.
\item [I2:] \label{ir1} Assume $A=z_1+z_2+z_3> \frac{7a+5b}4$, if $z_r< \frac{a+3b}4$ or $s_r+z_r< \frac{a+3b}2$, then the influence region contains $I_i$, $i=j-2\cdots,j+3$.
\item [I3:] \label{ir2} Assume $A=z_1+z_2+z_3< \frac{5a+7b}4$, if $z_\ell> \frac{3a+b}4$ or $s_\ell+z_\ell> \frac{3a+b}2$, then the influence region contains $I_i$, $i=j-3\cdots,j+2$.
\item [I4:] \label{ir3} In all other cases, the influence region contains $I_i$, $i=j-2,\cdots,j+2.$
\end{enumerate}
We say the ETC $I_j$ is isolated if $R_j$ does not overlap with influence regions of any other ETCs. 
\end{definition}

Now, we can state the main theorem.
\begin{theorem}\label{main}
Consider the hyperbolic conservation law \eqref{scalar1d} with $f(u)=\frac{u^2}2$, initial condition $u_0(x)$. Suppose the solution to problem \eqref{scalar1d} has compact support or satisfies periodic boundary conditions, $a\geq u_0(x)\geq b$, and the time step size satisfies
\begin{equation}\label{time}
\Delta t<\lambda\Delta x\quad\textrm{with}\quad\lambda =\frac{4}{a-b}.
\end{equation}
If the influence region of each ETC is given in Definition \ref{ir} and all the ETCs are isolated, then the numerical approximations obtained from the first-order EL FV scheme \eqref{mol_3_first} is TVD and MPP.
\end{theorem}
\begin{remark}
In the above theorem, we can choose $a$ and $b$ to be the local maximum and minimum in the influence region of the ETC and the two neighboring cells of the influence region. For example, if the influence region of the ETC $I_i$ contains $I_{i-2}$, $I_{i-1}$, $I_i$, $I_{i+1}$, $I_{i+2}$, then we take
$$
a_i=\max_{j=0,\pm1,\pm2,\pm3}\{u_{i+j}\},\quad b_i=\min_{j=0,\pm1,\pm2,\pm3}\{u_{i+j}\},
$$
and the conclusion is also correct under the condition
$$
\Delta t<\lambda\Delta x\quad\textrm{with}\quad\lambda=\frac{4}{\max_i\{a_i-b_i\}}.
$$
The proof is straightforward since if we merge 5 cells and the information to be used contains the cells in the influence region and the two neighboring cells. This can also be observed from the proof in Section \ref{secm2}. The time stepping constraint in \eqref{time} is at least twice as large as the standard CFL condition.
\end{remark}

\subsection{Proof of Theorem \ref{main}}

\subsubsection{The basic ideas and a simple case}\label{secidea}
In this subsection, we proceed to discuss the case when we merge four to six cells according to the Definition \ref{ir} for the influence region.
Most cases, we will merge five cells and suppose the numerical approximations in those five cells are $z_\ell$, $z_1$, $z_2$, $z_3$, $z_r$, respectively from left to right. Then after the merge, the only information to be used for time evolution is the average of these five values $\frac{z_\ell+z_1+z_2+z_3+z_r}5$. Therefore, we can redefine those five values according to the criteria (a)-(d) outlined in the introduction, such that no troubled cells exist in the influence region, keeping the original average value. Obviously, we cannot modify the characteristic speed at the boundary of the influence region, i.e. the values of $z_\ell$ or $z_r$, since we cannot modify the numerical fluxes at those two boundaries. Therefore, we will adjust $z_1$, $z_2$, $z_3$ to $r_1$, $r_2$, $r_3$, keeping their average, preserving the maximum-principle and reducing the original total variation. We will show adjustment strategies in finding $r_1$, $r_2$, $r_3$, so that no more troubled cells exist. From there, Lemma \ref{modify_Harten} can be applied to prove the TVD property. Finally, we will briefly discuss what if the ETCs are not isolated.


We first show a simple case to find $\{r_1, r_2, r_3\}$ when the ETC is a trouble cell of type IV in Theorem \ref{firstcase}.
\begin{theorem}\label{firstcase}
Suppose $I_j$ is a troubled cell of type IV and it is an ETC, then $I_{j-1}$ is a troubled cell of type V. Assume the numerical approximations on $I_i$, $i=j-2,\cdots,j+1$ are $z_1$, $z_2$, $z_3$, $z_4$, respectively, and $b\leq z_i\leq a$. Then $b\leq z_1\leq z_3\leq z_2\leq z_4\leq a.$ We can define $r_i=z_i$, $i=1,4$, and $r_2=z_3$, $r_3=z_2$. With the newly defined $\{r_i\}_{i=1}^4$,
we have 
$$b\leq r_1\leq r_2\leq\frac{a+b}2\leq r_3\leq r_4\leq a, \quad TV(z_1,z_2,z_3,z_4)\geq TV(r_1,r_2,r_3,r_4).$$ 
$I_{j-1}$ and $I_j$ are no longer troubled cells. If $I_{j+1}$ is a troubled cell and the numerical approximation on $I_{j+2}$ is $z_r$, then $z_4>z_r+\frac2\lambda$; moreover, it is not a troubled cell of type V. If $I_{j-2}$ is a troubled cell and the numerical approximation on $I_{j-3}$ is $z_\ell$, then $z_\ell>z_1+\frac2\lambda$; moreover it is not a troubled cell of type IV.
\end{theorem}
\begin{proof}
The proof follows from definition \ref{def} and is straightforward. The last two statements on trouble cells can be obtained following the same line; we will only prove the former statement for simplicity. 
Clearly, $I_{j+1}$ can only be a troubled cell of type I, III and V. If it is a troubled cell of type III and V, then $z_4>z_r+\frac2\lambda$ is automatically true. Otherwise, we have $r_4\geq r_3>z_r+\frac2\lambda$. Finally, if $I_{j+1}$ is a troubled cell of type V, then $\frac{a+b}2\leq r_3\leq z_r< z_4-\frac2\lambda\leq\frac{a+b}2$, which is a contradiction. 
\end{proof}
\begin{remark}
In the above theorem, the properties of the potential troubled cells on the boundaries of the influence region are for discussing the case of non-isolated ETCs.
\end{remark}

\subsubsection{Adjustment from $z_1$, $z_2$, $z_3$ to $\tilde{z}_1$, $\tilde{z}_2$, $\tilde{z}_3$ with TVD property}
\label{secm2}

When an ETC is not a troubled cell of type IV, the design and analysis of the adjustment strategy is rather involved. We will first identify some special values of $\{\tilde{z}_1, \tilde{z}_2, \tilde{z}_3\}$ provided in Table~\ref{mintv} to simplify the discussion of adjustment strategy for $\{r_1, r_2, r_3\}$ in the next subsection.

Suppose $I_j$ is an ETC for $\lambda$, the numerical approximations in the five merged cells and the two neighboring cells are $s_\ell$, $z_\ell$, $z_1$, $z_2$, $z_3$, $z_r$, $s_r$, respectively from left to right. Moreover, we denote the cells to be $I_{j-3},$ $\cdots,$ $I_{j+3}$ from left to right. 
The following lemma is used to find $r$'s with the minimum total variation.
\begin{lemma}
\label{lemma_r3}
Suppose $\max\{z_1,z_2\}\geq\max\{z_2,z_3\}$ and $r_1=\max\{z_1,z_2,z_3\}$, $r_3=\min\{z_1,z_2,z_3\}$, $r_2=z_1+z_2+z_3-r_1-r_3$, then $TV(z_\ell,r_1,r_2,r_3,z_r)\leq TV(z_\ell,z_1,z_2,z_3,z_r)$.
\end{lemma}
\begin{proof}
It is easy to see that $r_1\geq r_2\geq r_3$. There are several cases to be discussed.
\begin{enumerate}
\item $r_1=z_1$, then
$$
TV(z_\ell,r_1,r_2,r_3,z_r)=TV(z_\ell,z_1,r_3,z_r)\leq\max\{TV(z_\ell,z_1,z_2,z_r),TV(z_\ell,z_1,z_3,z_r)\}\leq TV(z_\ell,z_1,z_2,z_3,z_r)
$$
\item $r_1=z_2$. Since $I_j$ is an ETC, then by Definition \ref{def}, $I_j$ can be a troubled cell of type I, II or III. Therefore, $r_3\neq z_1$, leading to $r_3=z_3$, $r_2=z_1$ and $z_2\geq z_1\geq z_3$. Then we have
$$
TV(z_\ell,r_1,r_2,r_3,z_r)=TV(z_\ell,z_2,z_1,z_3,z_r)=TV(z_\ell,z_2,z_3,z_r)\leq TV(z_\ell,z_1,z_2,z_3,z_r).
$$
\end{enumerate}
\end{proof}
Since $I_j$ is an ETC, we will merge at least five cells and redefine the values of $z_1$, $z_2$ and $z_3$. Based Lemma \ref{lemma_r3}, we assume
$$
z_1\geq z_2\geq z_3,\quad z_1\geq z_3+\frac2\lambda.
$$
We first design the admissible set as
$$
G=\left\{(z_1,z_2,z_3):a\geq z_1\geq z_2\geq z_3\geq b,\ z_1\geq z_3+\frac2\lambda\right\}.
$$
The following lemmas will be used to find the possible upper and lower bounds of $z_1$ and $z_3$.
\begin{lemma}\label{compare0}
Suppose $(z_1,z_2,z_3)\in G$ and $\lambda=\frac4{a-b}$, then we have
$$
z_1\geq\frac{a+b}2,\quad z_3\leq\frac{a+b}2,\quad \frac{a+5b}2\leq A=z_1+z_2+z_3\leq \frac{5a+b}2.
$$
\end{lemma}
\begin{proof}
The proof follows from direct computation. Actually,
$$
z_1\geq z_3+\frac2\lambda\geq b+\frac{a-b}2=\frac{a+b}2.
$$
Similarly,
$$
z_3\leq z_1-\frac2\lambda\leq a-\frac{a-b}2=\frac{a+b}2.
$$
Finally,
$$
A=z_1+z_2+z_3\leq a+a+\frac{a+b}2=\frac{5a+b}2,
$$
and
$$
A=z_1+z_2+z_3\geq \frac{a+b}2+b+b=\frac{a+5b}2.
$$
\end{proof}
\begin{lemma}\label{ineq}
Suppose $\frac{a+5b}2\leq A\leq \frac{5a+b}2$ and $\lambda=\frac4{a-b}$, then $$
\frac13A+\frac2{3\lambda}\leq a,\quad \frac13A-\frac2{3\lambda}\geq b,\quad \frac13A-\frac4{3\lambda}\leq \frac{a+b}2=a-\frac2\lambda,\quad A-\frac{3a+b}2\leq a,\quad A-\frac{a+3b}2\geq b.
$$
\end{lemma}
\begin{proof}
The proof follows from direct computation, so we skip it.
\end{proof}

To ensure that the total variation does not increase, we aim to minimize $TV(z_\ell,z_1,z_2,z_3,z_r)$ in the admissible set $G$. Define $A=z_1+z_2+z_3$, we can choose $z_1$, $z_2$, $z_3$ in $G$ to achieve the minimum total variation $TV(z_\ell,z_1,z_2,z_3,z_r)$ following Table \ref{mintv}, where the conditions cover all the possibilities, yet may not be mutually exclusive, since we do not have the relationship between $z_\ell$ and $z_r$. The result is given in the following theorem.
\begin{table}[!htbp]
    \centering
    \begin{tabular}{|c|c|c|c|c|}\hline
Case&Conditions               &  $z_1$                    & $z_2$ & $z_3$\\\hline&&&&\\
\multirow{2}{*}{1}&$A\geq3z_\ell-\frac2\lambda$&\multirow{2}{*}{$\tilde{z}_1=z_r+\frac2\lambda$}&\multirow{2}{*}{$\tilde{z}_2=A-2z_r-\frac2\lambda$}&\multirow{2}{*}{$\tilde{z}_3=z_r$}\\
&$z_3\geq z_r$, $\frac13A-\frac4{3\lambda}\leq z_r$&&&\\[3ex]\hline&&&&\\
\multirow{2}{*}{2}&$A\geq3z_\ell-\frac2\lambda$&\multirow{2}{*}{$\tilde{z}_1=\frac{a+b}2$}&\multirow{2}{*}{$\tilde{z}_2=A-\frac{a+3b}2$}&\multirow{2}{*}{$\tilde{z}_3=b$}\\
&$z_3\leq z_r$, $A\leq a+2b$&&&\\[3ex]\hline&&&&\\
\multirow{2}{*}{3}&$A\geq3z_\ell-\frac2\lambda$&\multirow{2}{*}{$\tilde{z}_1=\frac13A+\frac2{3\lambda}$}&\multirow{2}{*}{$\tilde{z}_2=\frac13A+\frac2{3\lambda}$}&\multirow{2}{*}{$\tilde{z}_3=\frac13A-\frac4{3\lambda}$}\\
&all other possibilities&&&\\[3ex]\hhline{|=|=|=|=|=|}&&&&\\
\multirow{2}{*}{4}&$A\leq3z_r+\frac2\lambda$&\multirow{2}{*}{$\tilde{z}_1=z_\ell$}&\multirow{2}{*}{$\tilde{z}_2=A-2z_\ell+\frac2\lambda$}&\multirow{2}{*}{$\tilde{z}_3=z_\ell-\frac2\lambda$}\\
&$z_1\leq z_\ell$, $\frac13A+\frac4{3\lambda}\geq z_\ell$&&&\\[4ex]\hline&&&&\\
\multirow{2}{*}{5}&$A\leq3z_r+\frac2\lambda$&\multirow{2}{*}{$\tilde{z}_1=a$}&\multirow{2}{*}{$\tilde{z}_2=A-\frac{3a+b}2$}&\multirow{2}{*}{$\tilde{z}_3=\frac{a+b}2$}\\
&$z_1\geq z_\ell$, $A\geq 2a+b$&&&\\[3ex]\hline&&&&\\
\multirow{2}{*}{6}&$A\leq3z_r+\frac2\lambda$&\multirow{2}{*}{$\tilde{z}_1=\frac13A+\frac4{3\lambda}$}&\multirow{2}{*}{$\tilde{z}_2=\frac13A-\frac2{3\lambda}$}&\multirow{2}{*}{$\tilde{z}_3=\frac13A-\frac2{3\lambda}$}\\
&all other possibilities&&&\\[3ex]\hhline{|=|=|=|=|=|}&&&&\\
\multirow{2}{*}{7}&$3z_r+\frac2\lambda\leq A\leq 3z_\ell-\frac2\lambda$&$z_r+\frac2\lambda\leq \tilde{z}_1\leq \frac13A+\frac4{3\lambda}$&$\tilde{z}_2=A-2\tilde{z}_1+\frac2\lambda$&\multirow{2}{*}{$\tilde{z}_3=\tilde{z}_1-\frac2\lambda$}\\
&$z_\ell\geq z_r+\frac2\lambda$ &$\frac13A+\frac2{3\lambda}\leq \tilde{z}_1\leq z_\ell$&$=A-2\tilde{z}_3-\frac2\lambda$&\\[3ex]\hline&&&&\\
\multirow{2}{*}{8}&$3z_r+\frac2\lambda\leq A\leq 3z_\ell-\frac2\lambda$&\multirow{2}{*}{$\tilde{z}_1=z_r+\frac2\lambda$}&\multirow{2}{*}{$\tilde{z}_2=A-\tilde{z}_1-\tilde{z}_3$}&\multirow{2}{*}{$\tilde{z}_3=z_r$}\\
&$z_\ell\leq z_r+\frac2\lambda\leq a,\ z_1\geq z_\ell$ &&&\\[3ex]\hline&&&&\\
\multirow{2}{*}{9}&$3z_r+\frac2\lambda\leq A\leq 3z_\ell-\frac2\lambda$&\multirow{2}{*}{$\tilde{z}_1=a$}&\multirow{2}{*}{$\tilde{z}_2=A-\frac{3a+b}2$}&\multirow{2}{*}{$\tilde{z}_3=\frac{a+b}2$}\\
&$a\leq z_r+\frac2\lambda,\ z_1\geq z_\ell$ &&&\\[3ex]\hline&&&&\\
\multirow{2}{*}{10}&$3z_r+\frac2\lambda\leq A\leq 3z_\ell-\frac2\lambda$&\multirow{2}{*}{$\tilde{z}_1=z_\ell$}&\multirow{2}{*}{$\tilde{z}_2=A-\tilde{z}_1-\tilde{z}_3$}&\multirow{2}{*}{$\tilde{z}_3=z_\ell-\frac2\lambda$}\\
&$z_\ell\leq z_r+\frac2\lambda,\ z_1\leq z_\ell$&&&\\[3ex]\hline
    \end{tabular}
    \caption{Summary of the values of $\tilde{z}_1$, $\tilde{z}_2$, $\tilde{z}_3$ yielding minimum total variation in the admissible set $G$.}
    \label{mintv}
\end{table}
\begin{theorem}\label{thm_mintv}
Let $z_\ell,z_1,z_2,z_3,z_r\in[b,a]$ be the numerical approximations within five adjacent cells from left to right, such that $\frac{a+5b}2\leq A=z_1+z_2+z_3\leq \frac{5a+b}2$, then we can define $\tilde{z}_1,\tilde{z}_2,\tilde{z}_3$ in the admissible set $G$ following Table \ref{mintv} such that $TV(z_\ell,\tilde{z}_1,\tilde{z}_2,\tilde{z}_3,z_r)\leq TV(z_\ell,z_1,z_2,z_3,z_r)$. In addition, the chosen $\tilde{z}_1$, $\tilde{z}_2$, $\tilde{z}_3$ satisfy $\tilde{z}_1\geq\frac{a+b}2$, $\tilde{z}_3\leq\frac{a+b}2$ and $\tilde{z}_1=\tilde{z}_3+\frac2\lambda$. 
\end{theorem}
\begin{proof}
$\tilde{z}_1=\tilde{z}_3+\frac2\lambda$ can be observed from Table \ref{mintv} directly. $\tilde{z}_1\geq\frac{a+b}2$, $\tilde{z}_3\leq\frac{a+b}2$ can be obtained by Lemma \ref{compare0}. Now we proceed to demonstrate how to find the $\tilde{z}_i's$, $i=1,2,3$.
We consider different cases one by one and use Lemma \ref{ineq}.
\begin{itemize}
\item \underline{\bf{Cases 1, 2 and 3.}} Since $\displaystyle 3z_\ell-\frac2\lambda\leq A=z_1+z_2+z_3\leq 3z_1-\frac2\lambda,$ then
$$
z_1\geq \frac13A+\frac2{3\lambda}\geq z_\ell.
$$
Hence
$$
TV(z_\ell,z_1,z_2,z_3,z_r)=|z_\ell-z_1|+|z_1-z_3|+|z_3-z_r|=z_1-z_\ell+z_1-z_3+|z_3-z_r|.
$$
Then we have two possibilities:
\begin{enumerate}
\item If $z_3\leq z_r$, then
$$
TV(z_\ell,z_1,z_2,z_3,z_r)=z_1-z_\ell+z_1-z_3+z_r-z_3=2(z_1-z_3)-z_\ell+z_r\geq \frac4\lambda-z_\ell+z_r.
$$
\begin{enumerate}
\item If $\frac13A-\frac4{3\lambda}\geq b$, i.e. $A\geq a+2b$, then we can take
$$
\tilde{z}_1=\frac13A+\frac2{3\lambda}\in[z_\ell,z_1],\quad \tilde{z}_3=\frac13A-\frac4{3\lambda},\quad \tilde{z}_2=\frac13A+\frac2{3\lambda}\in[\tilde{z}_3,\tilde{z}_1]
$$
\begin{enumerate}
\item If $\frac13A-\frac4{3\lambda}\leq z_r$, then $\tilde{z}_3\leq z_r$. Therefore
$$
TV(z_\ell,\tilde{z}_1,\tilde{z}_2,\tilde{z}_3,z_r)=\tilde{z}_1-z_\ell+\tilde{z}_1-\tilde{z}_3+z_r-\tilde{z}_3=\frac4{\lambda}-z_\ell+z_r\leq TV(z_\ell,z_1,z_2,z_3,z_r).
$$
\item If $\frac13A-\frac4{3\lambda}\geq z_r$, then $\tilde{z}_3\geq z_r\geq z_3$. Therefore,
$$
TV(z_\ell,\tilde{z}_1,\tilde{z}_2,\tilde{z}_3,z_r)=\tilde{z}_1-z_\ell+\tilde{z}_1-\tilde{z}_3+\tilde{z}_3-z_r\leq z_1-z_\ell+z_1-z_3\leq TV(z_\ell,z_1,z_2,z_3,z_r),
$$
where in the first inequality we used the fact that $z_1\geq \tilde{z}_1$ and $z_3\leq z_r$.
\end{enumerate}
\item If $A\leq a+2b$, then $\frac13A-\frac4{3\lambda}\leq b$ leading to $\frac13A+\frac2{3\lambda}\leq \frac{a+b}2$. We take
$$
\tilde{z}_1=\frac{a+b}2\in[z_\ell,z_1],\quad \tilde{z}_3=b,\quad \tilde{z}_2=A-\tilde{z}_1-\tilde{z}_3=A-\frac{a+3b}2
$$
It is easy to see $\displaystyle \tilde{z}_3=b=\frac{a+5b}2-\frac{a+3b}2\leq \tilde{z}_2\leq a+2b-\frac{a+3b}2=\frac{a+b}2=\tilde{z}_1$. Moreover,
$$
TV(z_\ell,\tilde{z}_1,\tilde{z}_2,\tilde{z}_3,z_r)=\tilde{z}_1-z_\ell+\tilde{z}_1-\tilde{z}_3+z_r-\tilde{z}_3=\frac4\lambda-z_\ell+z_r\leq TV(z_\ell,z_1,z_2,z_3,z_r).
$$
\end{enumerate}

\item If $z_3\geq z_r$, then $\displaystyle z_1\geq z_r+\frac2\lambda,\  z_r\leq z_3\leq\frac{a+b}2=a-\frac2\lambda$ and
$$
TV(z_\ell,z_1,z_2,z_3,z_r)=z_1-z_\ell+z_1-z_3+z_3-z_r=2z_1-z_\ell-z_r.
$$
Therefore, we need to choose the smallest possible $z_1$.
\begin{enumerate}
\item If $\frac13A-\frac4{3\lambda}\geq z_r$, then we can take
$$
a\geq z_1\geq\tilde{z}_1=\frac13A+\frac2{3\lambda}\geq z_\ell,\quad \tilde{z}_3=\frac13A-\frac4{3\lambda}\geq z_r\geq b,\quad \tilde{z}_2=\frac13A+\frac2{3\lambda}\in[\tilde{z}_3,\tilde{z}_1].
$$
\item If $\frac13A-\frac4{3\lambda}\leq z_r\leq a-\frac2\lambda$, then we have
$$
z_1\geq z_3+\frac2\lambda\geq z_r+\frac2\lambda.
$$
\begin{enumerate}
\item If $z_r\leq \frac13A-\frac2{3\lambda}$, then we can choose
$$
a\geq z_1\geq \tilde{z}_1=z_r+\frac2\lambda\geq \frac13A+\frac2{3\lambda}\geq z_\ell,\quad \tilde{z}_3=z_r\geq b,\quad \tilde{z}_2=A-2z_r-\frac2\lambda.
$$
Since $\frac13A-\frac4{3\lambda}\leq z_r$, then
$$
\tilde{z}_2=A-2z_r-\frac2\lambda\leq 3z_r+\frac4\lambda-2z_r-\frac2\lambda=z_r+\frac2\lambda=\tilde{z}_1.
$$
Moreover, since $ z_r\leq \frac13A-\frac2{3\lambda}$, then
$$
\tilde{z}_2=A-2z_r-\frac2\lambda\geq 3z_r+\frac2\lambda-2z_r-\frac2\lambda=z_r=\tilde{z}_3.
$$
\item If $\displaystyle z_r> \frac13A-\frac2{3\lambda}$, we cannot find any $z_1$, $z_2$, $z_3$ in $G$. Actually, since $z_3\geq z_r$ and $z_1\geq z_3+\frac2\lambda$, then
$$
z_r\leq z_3\leq z_2=A-z_1-z_3\leq A-2z_r-\frac2\lambda< 3z_r+\frac2\lambda-2z_r-\frac2\lambda=z_r,
$$
which is a contradiction.
\end{enumerate}
\end{enumerate}
\end{enumerate}
\item\underline{\bf{Cases 4, 5 and 6.}} The proof for Cases 4, 5 and 6 can be obtained following that for Cases 1, 2 and 3 with some minor changes, so we skip it.
\item\underline{\bf{Case 7.}} If $\displaystyle 3z_r+\frac2\lambda\leq A\leq 3z_\ell-\frac2\lambda$ and $\displaystyle z_\ell\geq z_r+\frac2\lambda$. We take $\tilde{z}_1$ such that 
$$
z_r+\frac2\lambda\leq \tilde{z}_1\leq \frac13A+\frac4{3\lambda},\quad \frac13A+\frac2{3\lambda}\leq \tilde{z}_1\leq z_\ell,
$$
and
$\displaystyle \tilde{z}_3=\tilde{z}_1-\frac2\lambda$, $\displaystyle \tilde{z}_2=A-\tilde{z}_1-\tilde{z}_3.$ We will show that such a selection of $\tilde{z}_1$, $\tilde{z}_2$, $\tilde{z}_3$ is in $G$.
\begin{enumerate}
\item Since $\displaystyle3z_r+\frac2\lambda\leq A$, then $\displaystyle \frac13A+\frac4{3\lambda}\geq z_r+\frac2{3\lambda}+\frac4{3\lambda}= z_r+\frac2\lambda.$
Moreover, since $\displaystyle A\leq 3z_\ell-\frac2\lambda$, then $\displaystyle\frac13A+\frac2{3\lambda}\leq z_\ell$. In addition, it is easy to see that $\displaystyle\frac13A+\frac2{3\lambda}\leq \frac13A+\frac4{3\lambda}$ and based on the assumption that $\displaystyle z_\ell\geq z_r+\frac2\lambda$, we have the existence of $\tilde{z}_1$.
\item By direct computation, we have
$$
\tilde{z}_2=A-\tilde{z}_1-\tilde{z}_3=A-2\tilde{z}_1+\frac2\lambda\leq3\tilde{z}_1-\frac2\lambda-2\tilde{z}_1+\frac2\lambda=\tilde{z}_1,$$
and
$$
\tilde{z}_2=A-2\tilde{z}_1+\frac2\lambda\geq3\tilde{z}_1-\frac4\lambda-2\tilde{z}_1+\frac2\lambda=\tilde{z}_3.
$$
\end{enumerate}
Finally, since $\tilde{z}_1\leq z_\ell\leq a$ and $\tilde{z}_3=\tilde{z}_1-\frac2\lambda\geq z_r\geq b$. Therefore, we conclude that $\tilde{z}_1$, $\tilde{z}_2$, $\tilde{z}_3$ are chosen from $G$. Finally, we will show that such a selection yields the minimum total variation. Actually, since $\tilde{z}_1\leq z_\ell$ and $\tilde{z}_3=\tilde{z}_1-\frac2\lambda\geq z_r$, we have $z_\ell\geq \tilde{z}_1\geq \tilde{z}_2\geq \tilde{z}_3\geq z_r$. Clearly, the total variation is the minimum, which is $z_\ell-z_r.$
\item\underline{\bf{Cases 8 and 9.}} Since $z_1\geq z_\ell$, then $TV(z_\ell,z_1,z_2,z_3,z_r)=z_1-z_\ell+z_1-z_3+|z_3-z_r|.$
\begin{enumerate}
\item If $z_r+\frac2\lambda\leq a$, we take
$$
\tilde{z}_1=z_r+\frac2\lambda\in[z_\ell,a],\quad \tilde{z}_3=z_r\geq b,\quad \tilde{z}_2=A-\tilde{z}_1-\tilde{z}_3.
$$
We need to show that the selection is in $G$. Actually, since $A\geq 3z_r+\frac2\lambda$, then
$$
\tilde{z}_2=A-\tilde{z}_1-\tilde{z}_3=A-2z_r-\frac2\lambda\geq 3z_r+\frac2\lambda-2z_r-\frac2\lambda=z_r=\tilde{z}_3.
$$
Moreover, since $A\leq3z_\ell-\frac2\lambda$ and $z_\ell\leq z_r+\frac2\lambda$, then $A\leq3z_r+\frac4\lambda,$ leading to
$$
\tilde{z}_2=A-2z_r-\frac2\lambda\leq3z_r+\frac4\lambda-2z_r-\frac2\lambda=z_r+\frac2\lambda=\tilde{z}_1.
$$
Therefore, $(\tilde{z}_1,\tilde{z}_2,\tilde{z}_3)\in G$. Next, we consider the TV. It is easy to see that
$$
TV(z_\ell,\tilde{z}_1,\tilde{z}_2,\tilde{z}_3,z_r)=\tilde{z}_1-z_\ell+\tilde{z}_1-\tilde{z}_3=z_r-z_\ell+\frac4\lambda.
$$
\begin{enumerate}
\item If $z_3\leq z_r$, then
$$
TV(z_\ell,z_1,z_2,z_3,z_r)=z_1-z_\ell+z_1-z_3+z_r-z_3=2(z_1-z_3)-z_\ell+z_r\geq \frac4\lambda-z_\ell+z_r=TV(z_\ell,\tilde{z}_1,\tilde{z}_2,\tilde{z}_3,z_r).
$$
\item If $z_3\geq z_r$, then $z_1\geq z_3+\frac2\lambda\geq z_r+\frac2\lambda.$ Therefore,
$$
TV(z_\ell,z_1,z_2,z_3,z_r)=z_1-z_\ell+z_1-z_3+z_3-z_r=2z_1-z_\ell-z_r\geq \frac4\lambda-z_\ell+r_r=TV(z_\ell,\tilde{z}_1,\tilde{z}_2,\tilde{z}_3,z_r).
$$
\end{enumerate}
\item If $z_r+\frac2\lambda\geq a$, then $z_r\geq\frac{a+b}2\geq z_3$. Hence $$
TV(z_\ell,z_1,z_2,z_3,z_r)=z_1-z_\ell+z_1-z_3+z_r-z_3\geq\frac4\lambda+z_r-z_\ell.
$$
We take
$$
\tilde{z}_1=a,\quad \tilde{z}_3=\frac{a+b}2,\quad \tilde{z}_2=A-\frac{3a+b}2\leq a=\tilde{z}_1.
$$
Moreover, $\displaystyle \tilde{z}_2=A-\frac{3a+b}2\geq3z_r+\frac2\lambda-\frac{3a+b}2\geq\frac{3a+3b}2+\frac{a-b}2-\frac{3a+b}2=\frac{a+b}2=\tilde{z}_3.$ Therefore, $(\tilde{z}_1,\tilde{z}_2,\tilde{z}_3)\in G$. Finally,
$$
TV(z_\ell,\tilde{z}_1,\tilde{z}_2,\tilde{z}_3,z_r)=\tilde{z}_1-z_\ell+\tilde{z}_1-\tilde{z}_3+z_r-\tilde{z}_3=\frac4\lambda+z_r-z_\ell\leq TV(z_\ell,z_1,z_2,z_3,z_r).
$$
\end{enumerate}
\item\underline{\bf{Case 10.}} Since $z_1\leq z_\ell$ and $z_\ell\leq z_r+\frac2\lambda$, then
$$z_3\leq z_1-\frac2\lambda\leq z_\ell-\frac2\lambda\leq z_r+\frac2\lambda-\frac2\lambda=z_r.
$$
Hence,
$$
TV(z_\ell,z_1,z_2,z_3,z_r)=z_\ell-z_1+z_1-z_3+z_r-z_3=z_\ell+z_r-2z_3\geq z_\ell+z_r-2(z_\ell-\frac2\lambda)=\frac4\lambda+z_r-z_\ell.
$$
We take
$$
\tilde{z}_1=z_\ell\leq a,\quad \tilde{z}_3=z_\ell-\frac2\lambda\geq z_1-\frac2\lambda\geq b,\quad \tilde{z}_2=A-\tilde{z}_1-\tilde{z}_3.
$$
We first show that such a selection is in $G$. Since $\displaystyle A\leq3z_l-\frac2\lambda$, then
$$
\tilde{z}_2=A-\tilde{z}_1-\tilde{z}_3\leq 3z_l-\frac2\lambda-\tilde{z}_1-\tilde{z}_1+\frac2\lambda=\tilde{z}_1.
$$
On the other hand, since $\displaystyle A\geq 3z_r+\frac2\lambda$ and $\displaystyle z_\ell\leq z_r+\frac2\lambda$, then $\displaystyle A\geq3z_\ell-\frac6\lambda+\frac2\lambda=3z_\ell-\frac4\lambda.$
Therefore,
$$
\tilde{z}_2=A-\tilde{z}_1-\tilde{z}_3\geq3z_\ell-\frac4\lambda-z_\ell-z_\ell+\frac2\lambda=z_\ell-\frac2\lambda=\tilde{z}_3.
$$
Finally, since $\displaystyle z_r\geq z_\ell-\frac2\lambda=\tilde{z}_3,$ then
$$
TV(z_\ell,\tilde{z}_1,\tilde{z}_2,\tilde{z}_3,z_r)=\tilde{z}_1-\tilde{z}_3+z_r-\tilde{z}_3=z_\ell-z_\ell+\frac2\lambda+z_r-z_\ell+\frac2\lambda=\frac4\lambda+z_r-z_\ell\leq TV(z_\ell,z_1,z_2,z_3,z_r).
$$
\end{itemize}
\end{proof}

\subsubsection{Adjustment from $\tilde{z}_1$, $\tilde{z}_2$, $\tilde{z}_3$ to $r_1$, $r_2$, $r_3$ with TVD property}

 We summarized all the possible cases that yielding minimum total variation in $G$ in Table~\ref{mintv}. Now we only need to discuss the $\tilde{z}_i's$, $i=1,2,3$, given in Table \ref{mintv}. We can proceed to redefine the numerical approximations within the influence region such that the new total variation is no larger than those given in Theorem \ref{thm_mintv}. We assume the redefined values are $r_1$, $r_2$ and $r_3$, respectively. In order to simplify the discussion, we always assume
$$
r_1\geq r_2\geq r_3.
$$
Before we state the next lemma, we would like to define a negative fraction as $+\infty$ in the rest of the paper.
\begin{lemma}\label{time0}
Suppose $r_1$, $r_2$, $r_3$ are the modified numerical approximations in cells $I_{j-i},$ $I_j,$ $I_{j+1}$, then the partition lines originated from $x_{i-\frac12}$, $i=j-2,\cdots,j+3$ do not intersect under the condition
$$
\tilde{\lambda} < 2\min\left\{\frac1{s_\ell-r_1}, \frac1{z_\ell-r_2}, \frac1{r_1-r_3}, \frac1{r_2-z_r}, \frac1{r_3-s_r}\right\}.
$$
\end{lemma}
\begin{proof}
Suppose the partition lines originated from $x_{j\pm\frac12}$ do not intersect, then we have $\displaystyle \frac{r_1+r_2}2\Delta t<\Delta x+\frac{r_2+r_3}2\Delta t,$ which further implies $\displaystyle (r_1-r_3)\Delta t<2\Delta x.$ Clearly, if $r_1-r_3\leq0$, then the two characteristics never intersect, and the time should be $+\infty.$ Otherwise, we need $\displaystyle \Delta t<\frac{2\Delta x}{r_1-r_3}$, and a sufficient condition is $\displaystyle\tilde{\lambda}<\frac{2}{r_1-r_3}$. The other terms on the right-hand side can be obtained following the same lines.
\end{proof}
The above lemma has the following straightforward corollary.
\begin{corollary}\label{time_useful}
Suppose the numerical approximations are within the interval $[b,a]$, if \begin{equation}\label{rs}
r_1\geq\frac{a+b}2,\quad r_3\leq\frac{a+b}2,\quad r_1-r_3\leq\frac2\lambda,\quad z_\ell-r_2\leq\frac2\lambda,\quad r_2-z_r\leq\frac2\lambda,
\end{equation}
then $\displaystyle s_\ell-r_1\leq\frac2\lambda,\ r_3-s_r\leq\frac2\lambda$ and the partition lines originated from $x_{i-\frac12}$, $i=j-2,\cdots,j+3$ do not intersection under the condition \eqref{time}.
\end{corollary}
\begin{proof}The two inequalities are straightforward, hence we skip the details. We first consider the three missing terms in Lemma \ref{time0} one by one.
\begin{enumerate}
\item $\displaystyle\frac1{s_\ell-r_1}.$
\begin{enumerate}
\item If $s_\ell\leq r_1$, then $\displaystyle\frac1{s_\ell-r_1}=\infty$, and it does not contribute anything in the time step restriction.
\item If $s_\ell\geq r_1$, then
$$
s_\ell-r_1\leq a-\frac{a+b}2=\frac{a-b}2.
$$
Since $\lambda=\frac4{a-b}$, then $\displaystyle\frac1{s_\ell-r_1}$ does not contribute anything in the time step restriction.
\end{enumerate}
\item $\displaystyle\frac1{r_3-s_r}.$ The analysis can be obtained following the same lines given for $\displaystyle\frac1{s_\ell-r_1}$, so we skip it.
\item $\displaystyle\frac1{r_1-r_3}.$ Since $\displaystyle 0\leq r_1-r_3\leq \frac2\lambda$, then $\displaystyle\frac1{r_1-r_3}\geq\frac\lambda2$. Therefore, $\displaystyle\frac1{r_1-r_3}$ does not contribute anything in the time step restriction.
\item $\displaystyle\frac1{z_\ell-r_2}$ and $\displaystyle\frac1{r_2-z_r}$. The analyses can be obtained following the same lines given for $\displaystyle\frac1{r_1-r_3}.$ So we skip them.
\end{enumerate}
\end{proof}
The above corollary is quite helpful in finding suitable numerical approximations within the influence region such that the partition lines do not intersect among the cases in Table \ref{mintv}. However, two choices of $\tilde{z}_i's$ needs special treatments. We state the technique for the two special cases in Theorem \ref{thm_redefine_special}-\ref{thm_redefine_special2} and the summary for all the cases in Theorem \ref{thm_redefine}.

We would like to start from some special cases in Table \ref{mintv}. For simplicity of presentation, we drop the tilde and use $z_i$ for $\tilde{z}_i$, $i=1,2,3$.
\begin{theorem}\label{thm_redefine_special}
Suppose $I_j$ is an ETC, and the influence region is given by Case \ref{ir3} in Definition \ref{ir}. If $\lambda=\frac4{a-b}$, then we can find $r_1$, $r_2$, $r_3$ defined in cells $I_{j-1},$ $I_j$, $I_{j+1}$, respectively, such that
$$
a\geq r_1\geq r_2\geq r_3\geq b,\quad s_\ell-r_1\leq\frac2\lambda,\quad z_\ell-r_2\leq\frac2\lambda,\quad r_1-r_3\leq\frac2\lambda,\quad r_2-z_r\leq\frac2\lambda,\quad  r_3-s_r\leq\frac2\lambda
$$
under one of the following conditions 
\begin{enumerate}
\item $\displaystyle z_1=z_2=\frac13A+\frac2{3\lambda}$, $\displaystyle z_3=\frac13A-\frac4{3\lambda}$, with $\displaystyle \frac{7a+5b}4\geq A\geq a+2b$.
\item $\displaystyle z_1=z_2=\frac13A+\frac2{3\lambda}$, $\displaystyle z_3=\frac13A-\frac4{3\lambda}$, with $\displaystyle \frac{7a+5b}4\leq A\leq\frac{5a+b}2$, $\displaystyle z_r\geq \frac{a+3b}4$ and $\displaystyle s_r+z_r\geq \frac{a+3b}2$.
\item $\displaystyle z_1=\frac13A+\frac4{3\lambda}$, $\displaystyle z_2=z_3=\frac13A-\frac2{3\lambda}$, with $\displaystyle \frac{5a+7b}4\leq A\leq 2a+b$.
\item $\displaystyle z_1=\frac13A+\frac4{3\lambda}$, $\displaystyle z_2=z_3=\frac13A-\frac2{3\lambda}$, with $\displaystyle \frac{5a+7b}4\geq A\geq \frac{a+5b}2$, $\displaystyle z_\ell\leq\frac{3a+b}4$ and $\displaystyle s_\ell+z_\ell\leq\frac{3a+b}2$.
\end{enumerate}
In addition, the proposed numerical approximations satisfy
$$
\sum_{j=1,2,3}z_j=\sum_{j=1,2,3}r_j,\quad  TV(z_\ell,z_1,z_2,z_3,z_r)\geq TV(z_\ell,r_1,r_2,r_3,z_r).
$$
Moreover, after updating the numerical approximations, $I_i$, $i=j-1, \cdots, j+1$ are not troubled cells. If $I_{j-2}$ is a troubled cell, it can only be a troubled cell of type II or IV. If $I_{j+2}$ is a troubled cell, it can only be a troubled cell of type III or V.
\end{theorem}
\begin{proof}
We consider conditions 1 and 2 only, and the proof for conditions 3 and 4 can be obtained following the same line. We verify the conditions in Corollary \ref{time_useful}.

\begin{enumerate}
\item Suppose $\displaystyle a+2b\leq A\leq \frac{7a+5b}4$, then
$$
z_2+z_3=\frac13A+\frac2{3\lambda}+\frac13A-\frac4{3\lambda}=\frac23A-\frac2{3\lambda}\leq\frac{7a+5b}6-\frac{a-b}6=a+b,
$$
and
$$
z_2=\frac13A+\frac2{3\lambda}\geq\frac{a+2b}3+\frac{a-b}6=\frac{a+b}2.
$$
We take $r_1=z_1=z_2\geq\frac{a+b}2$, $r_2=\frac{a+b}2$, $r_3=z_2+z_3-r_2\in [z_3,r_2]$, to get
$$
\sum_{j=\ell,1,2,3,r}z_j=\sum_{j=\ell,1,2,3,r}r_j.
$$
Moreover, following direct computation, we can get
$$
b\leq z_3\leq r_3=z_2+z_3-r_2\leq\frac{a+b}2\leq z_1,\quad r_1-r_3\leq z_1-z_3=\frac2\lambda,
$$
and
$$
z_\ell-r_2\leq a-\frac{a+b}2=\frac{a-b}2=\frac2\lambda,\quad r_2-z_r\leq \frac{a+b}2-b=\frac{a-b}2=\frac2\lambda.
$$
Finally,
$$
TV(z_\ell,r_1,r_2,r_3,z_r)=TV(z_\ell,z_1,r_3,z_r)\leq TV(z_\ell,z_1,r_3,z_3,z_r)=TV(z_\ell,z_1,z_3,z_r)=TV(z_\ell,z_1,z_2,z_3,z_r).
$$
\item Suppose $\displaystyle \frac{7a+5b}4\leq A\leq\frac{5a+b}2$. Then $\displaystyle z_2+z_3=\frac23A-\frac2{3\lambda}\geq\frac23\frac{7a+5b}4-\frac23\frac{a-b}4=a+b$, leading to $\displaystyle z_2\geq\frac{a+b}2,$ and $z_\ell-z_2\leq a-\frac{a+b}2=\frac2\lambda$. 
We use the following procedure to find $r_j's$, $j=1,2,3$.
\begin{enumerate}
\item If $z_2\leq z_r+\frac2\lambda$, then we can take $r_j=z_j$, $j=1,2,3$.
\item If $z_2> z_r+\frac2\lambda$, then we take $r_1=z_1,$ $r_2=z_r+\frac2\lambda$ and $r_3=z_2+z_3-r_2.$ Clearly, we have $r_2<z_2=r_1$ and $r_3>z_3$, leading to $r_1-r_3<\frac2\lambda$ and $z_\ell-r_2\leq a-z_r-\frac2\lambda\leq a-b-\frac2\lambda=\frac2\lambda.$ Since
$$
z_2+z_3=\frac23A-\frac2{3\lambda}\leq\frac{5a+b}3-\frac{a-b}6=\frac{3a+b}2=\frac{a+3b}2+(a-b)\leq 2z_2+\frac4\lambda=2r_2,
$$
then $r_3=z_2+z_3-r_2\leq r_2.$
Moreover,
$$
r_3-s_r=z_2+z_3-r_2-s_r=z_2+z_3-z_r-\frac2\lambda-s_r\leq a+\frac{a+b}2-\frac{a-b}2-\frac{a+3b}2=\frac{a-b}2=\frac2\lambda,
$$
and
$$
TV(z_\ell,r_1,r_2,r_3,z_r)=TV(z_\ell,z_1,r_3,z_r)\leq TV(z_\ell,z_1,r_3,z_3,z_r)=TV(z_\ell,z_1,z_3,z_r)=TV(z_\ell,z_1,z_2,z_3,z_r).
$$
\end{enumerate}
\item Finally, based on the construction of $r_1,$ $r_2$, $r_3$, we can see that $I_{j-1}$, $I_j$, $I_{j+1}$ are not troubled cells. Since $r_1\geq\frac{a+b}2$, then $I_{j-2}$ can only be troubled cells of type II or IV. Finally, since $r_3\leq \frac{a+b}2$, then $I_{j+2}$ can only be a troubled cell of type III and V.
\end{enumerate}
\end{proof}
\begin{theorem}\label{thm_redefine_special2}
Suppose $I_j$ is an ETC, and the influence region is given by Case I2 in Definition \ref{ir}. If $\lambda=\frac4{a-b}$, then we can find $r_1$, $r_2$, $r_3$, $r_r$ defined in cells $I_{j-1},$ $I_j$, $I_{j+1}$, $I_{j+2}$, respectively, such that
$$
a\geq r_1\geq r_2\geq r_3\geq b,\quad r_1\geq\frac{a+b}2,\quad z_\ell-r_2\leq\frac2\lambda,\quad r_1-r_3\leq\frac2\lambda,\quad r_2-r_r\leq\frac2\lambda,\quad  r_3-s_r\leq\frac2\lambda,\quad b\leq r_r\leq\frac{a+b}2
$$
under the conditions that
$$
z_1=z_2=\frac13A+\frac2{3\lambda},\quad z_3=\frac13A-\frac4{3\lambda},\quad\frac{7a+5b}4< A\leq\frac{5a+b}2,
$$
and one of the following two inequalities
$$
s_r+z_r< \frac{a+3b}2,\quad z_r< \frac{a+3b}4.
$$
In addition, the proposed numerical approximations satisfy
$$
\sum_{j=1,2,3,r}z_j=\sum_{j=1,2,3,r}r_j,\quad  TV(z_\ell,z_1,z_2,z_3,z_r,s_r)\geq TV(z_\ell,r_1,r_2,r_3,r_r,s_r).
$$
Moreover, after updating the numerical approximations, $I_i$, $i=j-1, \cdots, j+2$ are not troubled cells. If $I_{j+3}$ is a troubled cell, it can only be a troubled cell of type III or V. If $I_{j-2}$ is a troubled cell, it can only be a troubled cell of type II and IV.
Similarly, by choosing the influence region as Case I2 in Definition \ref{ir}, we can find $r_\ell$, $r_1$, $r_2$, $r_3$ defined in cells $I_{j-2}$, $I_{j-1},$ $I_j$, $I_{j+1}$, respectively, such that
$$
a\geq r_1\geq r_2\geq r_3\geq b,\quad a\geq r_\ell\geq\frac{a+b}2,\quad s_\ell-r_1\leq\frac2\lambda,\quad r_\ell-r_2\leq\frac2\lambda,\quad r_1-r_3\leq\frac2\lambda,\quad r_2-z_r\leq\frac2\lambda,\quad r_3\leq\frac{a+b}2
$$
under the conditions that
$$
z_1=\frac13A+\frac4{3\lambda},\quad z_2=z_3=\frac13A-\frac2{3\lambda},\quad\frac{5a+7b}4> A\geq\frac{a+5b}2.
$$
and one of the following two inequalities
$$
z_\ell> \frac{3a+b}4,\quad s_\ell+z_\ell> \frac{3a+b}2.
$$
In addition, the proposed numerical approximations satisfy
$$
\sum_{j=\ell,1,2,3}z_j=\sum_{j=\ell,1,2,3}r_j,\quad  TV(s_\ell,z_\ell,z_1,z_2,z_3,z_r)\geq TV(s_\ell,r_\ell,r_1,r_2,r_3,z_r).
$$
Moreover, after updating the numerical approximations, $I_i$, $i=j-2, \cdots, j+1$ are not troubled cells. If $I_{j+2}$ is a troubled cell, it can only be a troubled cell of type III or V. If $I_{j-3}$ is a troubled cell, it can only be a troubled cell of type II and IV.
\end{theorem}
\begin{proof}
We consider Case I2 in Definition \ref{ir} only, and the proof for Case I3 can be obtained following the same lines with some minor changes. We construct the $r_i's$, $i=1,2,3,r$ as follows.
\begin{enumerate}
\item Define $\tilde{z}_j=z_j$, $j=1,r$ and $\tilde{z}_3=\frac{a+b}2\geq z_3$, $\tilde{z}_2=z_2+z_3-\tilde{z}_3>\frac{a+b}2$. Clearly, $z_2\geq\tilde{z}_2\geq\tilde{z}_3\geq z_3$. Moreover, we define $r_j=\tilde{z}_j$, $j=1,3$.
\item If $\tilde{z}_2\geq \frac{a+b}2+\frac1\lambda$ and $\tilde{z}_r\leq\frac{a+b}2-\frac1\lambda$, define $r_2=\tilde{z}_2-\frac1\lambda$ and $r_r=\tilde{z}_r+\frac1\lambda.$ Clearly,
$$
a\geq r_1=z_1=z_2\geq\tilde{z}_2\geq r_2\geq \frac{a+b}2=r_3=\tilde{z}_3\geq r_r\geq \tilde{z}_r=z_r\geq b.
$$
By Lemma \ref{compare0}, we know $r_1=z_1\geq\frac{a+b}2$. Since $r_3=\frac{a+b}2$, we have
$$
r_1-r_3\leq a-\frac{a+b}2=\frac{a+b}2\leq\frac2\lambda,\quad r_3-s_r\leq\frac{a+b}2-b=\frac{a-b}2=\frac2{\lambda}.
$$
Since $r_2\geq \frac{a+b}2$, then $z_\ell-r_2\leq a-\frac{a+b}2=\frac{a+b}2\leq\frac2\lambda$. Moreover,
$$
r_2-r_r=\tilde{z}_2-\frac1\lambda-\tilde{z}_r-\frac1\lambda\leq z_2-z_r-\frac2\lambda\leq a-b-\frac2\lambda=\frac2\lambda.
$$
Clearly,
$$\sum_{j=\ell,1,2,3,r}z_j=\sum_{j=\ell,1,2,3,r}\tilde{z}_j=\sum_{j=\ell,1,2,3,r}r_j.
$$
Finally,
\begin{align*}
&TV(z_\ell,z_1,z_2,z_3,z_r,s_r)=TV(z_\ell,z_1,z_2,\tilde{z}_2,\tilde{z}_3,z_3,z_r,s_r)\geq TV(z_\ell,\tilde{z}_1,\tilde{z}_2,\tilde{z}_3,\tilde{z}_r,s_r)\\
=&TV(z_\ell,r_1,\tilde{z}_2,r_2,r_3,r_r,\tilde{z}_r,s_r)\geq TV(z_\ell,r_1,r_2,r_3,r_r,s_r).
\end{align*}
\item If $\tilde{z}_2\leq \frac{a+b}2+\frac1\lambda$ and $\tilde{z}_r\leq\frac{a+b}2-\frac1\lambda$, define $r_2=\frac{a+b}2$ and $r_r=\tilde{z}_2+\tilde{z}_r-r_2$. Clearly,
$$
\tilde{z}_2\geq r_2=\frac{a+b}2= r_3=\tilde{z}_3\geq r_r\geq \tilde{z}_r=z_r\geq b.
$$
Moreover, $r_2-r_r\leq \frac{a+b}2-b=\frac{a-b}2\leq\frac2\lambda.$
All the other inequalities are exactly the same as in case 2.
\item If $\tilde{z}_2\geq \frac{a+b}2+\frac1\lambda$ and $\tilde{z}_r\geq\frac{a+b}2-\frac1\lambda$, define $r_r=\frac{a+b}2$ and $r_2=\tilde{z}_2+\tilde{z}_r-r_r$. Clearly,
$$
\tilde{z}_2\geq r_2\geq r_3=\tilde{z}_3=\frac{a+b}2= r_r\geq \tilde{z}_r\geq b.
$$
Moreover, $r_2-r_r\leq a-\frac{a+b}2=\frac{a-b}2\leq\frac2\lambda.$
All the other inequalities are exactly the same as in case 2.
\item If $\tilde{z}_2\leq \frac{a+b}2+\frac1\lambda$ and $\tilde{z}_r\geq\frac{a+b}2-\frac1\lambda$, define $r_2=\tilde{z}_2$ and $r_r=\tilde{z}_r$. Clearly,
$$
\tilde{z}_2= r_2\geq r_3=\tilde{z}_3\geq r_r= \tilde{z}_r,\quad \frac{a+b}2\geq\frac{a+3b}2-s_r> r_r=z_r\geq b.
$$
Here we use the fact that $\tilde{z}_r=z_r<\frac{a+b}2$. Moreover, $r_2-r_r\leq\frac{a+b}2+\frac1\lambda-\frac{a+b}2+\frac1\lambda=\frac2\lambda.$
All the other inequalities are exactly the same as in the above case.
\item Finally, based on the construction given above and $r_r\leq\frac{a+b}2$, it is easy to see $I_i$, $i=j-1,\cdots,j+2$ are not troubled cells after the modification, Since $r_1\geq\frac{a+b}2$, and $r_r\leq\frac{a+b}2$, then $I_{j-2}$ can only be a troubled cell of type II and IV and $I_{j+3}$ can only be a troubled cell of type III and V.
\end{enumerate}
\end{proof}

\begin{theorem}\label{thm_redefine}
Suppose the numerical approximations are within the interval $[b,a]$, and $I_j$ is an ETC. The numerical approximations on $I_{j-3},\cdots,I_{j+3}$ are $s_\ell$, $z_\ell$, $z_1$, $z_2$, $z_3$, $z_r$, $s_r$, respectively, with $z_1\geq z_2\geq z_3$ and $z_1\geq z_3+\frac2\lambda$. The influence region is given in Definition \ref{ir}. If we take $\lambda=\frac4{a-b}$, then we can find $r_\ell,$ $r_1$, $r_2$, $r_3$, $r_r$ defined in cells $I_{j-2},$ $\cdots,$ $I_{j+2}$, respectively, without changing the numerical approximations on the boundary cells in the influence region, such that the proposed new numerical approximations satisfy
$$
\sum_{j=\ell,1,2,3,r}z_j=\sum_{j=\ell,1,2,3,r}r_j,\quad TV(s_\ell,z_\ell,z_1,z_2,z_3,z_r,s_r)\geq TV(s_\ell,r_\ell,r_1,r_2,r_3,r_r,s_r).
$$
Moreover, one of the possible troubled cell in the influence region is the one on the right boundary and it can only be a troubled cell of type III or V. The other possible troubled cell in the influence region is the one on the left boundary and it can only be a troubled cell of type II or IV. Finally, the partition lines at the boundaries of the influence region keep the same.
\end{theorem}
\begin{proof}
We have demonstrated some partial results in Theorems \ref{thm_redefine_special} and \ref{thm_redefine_special2}. Now we will show that if the influence region is given as Case I4 in Definition \ref{ir}, then the adjusted numerical approximations satisfy
$$
a\geq r_1\geq r_2\geq r_3\geq b,\quad r_1\geq\frac{a+b}2,\quad  r_3\leq\frac{a+b}2,\quad r_1-r_3\leq\frac2\lambda,\quad z_\ell-r_2\leq\frac2\lambda,\quad r_2-z_r\leq\frac2\lambda.
$$
We discuss the cases given in Table \ref{mintv}.
\begin{itemize}
\item \underline{\bf{Case 1.}} We take $r_i=z_i$, with $i=1,2,3$. Then we only need to show $z_\ell-r_2\leq\frac2\lambda,\ r_2-z_r\leq\frac2\lambda$. Actually,
$$
r_2-z_r=z_2-z_r=z_2-z_3\leq z_1-z_3=\frac2\lambda.
$$
Finally, since $3z_\ell\leq A+\frac2\lambda\leq 3z_1$, then
$$
z_\ell-r_2=z_\ell-z_2\leq z_1-z_2\leq z_1-z_3\leq\frac2\lambda.
$$
\item \underline{\bf{Case 2.}} We take $r_i=z_i$, with $i=1,2,3$. It is easy to see that $z_1\geq z_\ell$. We follow the same analysis in Case 1 except
$$
r_2-z_r=z_2-z_r\leq z_1-z_r\leq\frac{a+b}2-b\leq\frac{a-b}2=\frac2\lambda.
$$
\item \underline{\bf{Case 3.}} By Lemma \ref{compare0} and Theorems \ref{thm_redefine_special} and \ref{thm_redefine_special2}, we only assume $\displaystyle \frac{a+5b}2\leq A< a+2b$, and take $r_i=z_i$, with $i=1,2,3$. We follow the same analysis in Case 1 except
$$
r_2-z_r=\frac13A+\frac2{3\lambda}-z_r\leq \frac{a+2b}3+\frac2{3\lambda}-b=\frac{a-b}3+\frac2{3\lambda}\leq\frac4{3\lambda}+\frac2{3\lambda}=\frac2\lambda.
$$
\item \underline{\bf{Cases 4, 5 and 6.}} The proof is similar to Cases 1, 2 and 3, so we skip it.
\item \underline{\bf{Case 7.}} We consider three possibilities.
\begin{enumerate}
\item If $\displaystyle \frac13A+\frac2{3\lambda}\geq z_r+\frac2\lambda$, then we can take $\displaystyle z_1=\frac13A+\frac2{3\lambda}$, $\displaystyle z_2=\frac13A+\frac2{3\lambda}$, $\displaystyle z_3=\frac13A-\frac4{3\lambda}$. Since $\displaystyle A\geq 3z_r+\frac4\lambda\geq 3b+a-b=a+2b,$ the conclusion follows from Theorems \ref{thm_redefine_special} and \ref{thm_redefine_special2}.
\item If $\displaystyle \frac13A+\frac4{3\lambda}\leq z_\ell$, then we can take
$\displaystyle z_1=\frac13A+\frac4{3\lambda}$, $\displaystyle z_2=\frac13A-\frac2{3\lambda}$, $\displaystyle z_3=\frac13A-\frac2{3\lambda}$. Since $\displaystyle A\leq 3z_\ell-\frac4\lambda\leq 3a-(a-b)=2a+b,$ the conclusion follows from Theorems \ref{thm_redefine_special} and \ref{thm_redefine_special2}.
\item If $\displaystyle \frac13A+\frac2{3\lambda}\leq z_r+\frac2\lambda$ and $\displaystyle \frac13A+\frac4{3\lambda}\geq z_\ell$, then
we can take $\displaystyle z_1=z_r+\frac2\lambda$, $\displaystyle z_3=z_r$, $\displaystyle z_2=A-2z_r-\frac2\lambda$, which returns to case 1. We can following the same proof to obtain all the required inequalities except $\displaystyle r_\ell-z_r\leq\frac2\lambda$, where a special condition was used for $z_\ell$ in case 1. Based on the assumptions, we have
$$
\frac13A+\frac2{3\lambda}\leq z_r+\frac2\lambda= z_1\leq z_\ell\leq\frac13A+\frac4{3\lambda},
$$
which further implies
$$
z_\ell+2z_r\leq 3\cdot\left(\frac13A+\frac4{3\lambda}\right)-2\cdot\frac2\lambda=A.
$$
Then we have
$$
z_\ell-r_2=z_\ell-z_2=z_\ell-\left(A-2z_r-\frac2\lambda\right)=z_\ell-A+2z_2+\frac2\lambda\leq\frac2\lambda.
$$
\end{enumerate}
\item \underline{\bf{Case 8.}} We take $r_i=z_i$, with $i=1,2,3$.  We follow the same analysis in Case 1 except
$$
z_\ell-r_2=z_\ell-z_2\leq z_r+\frac2\lambda-z_3=\frac2\lambda.
$$
\item \underline{\bf{Case 9.}} We take $r_i=z_i$, with $i=1,2,3$.  We follow the same analysis in Case 1 except
$$
z_\ell-r_2=z_\ell-z_2\leq z_1-z_2\leq\frac2\lambda,
$$
and
$$
r_2-z_r=z_2-z_r=A-\frac{3a+b}2-z_r\leq3z_\ell-\frac2\lambda-\frac{3a+b}2-a+\frac2\lambda\leq 2a-\frac{3a+b}2=\frac2\lambda.
$$
\item \underline{\bf{Case 10.}} We take $r_i=z_i$, with $i=1,2,3$.  We follow the same analysis in Case 1 except
$$
z_\ell-r_2=z_1-z_2\leq z_1-z_3=\frac2\lambda,
$$
and
$$
r_2-z_r=z_2-z_r\leq z_2-z_\ell+\frac2\lambda\leq z_1-z_\ell+\frac2\lambda=z_\ell-z_\ell+\frac2\lambda=\frac2\lambda.
$$
\end{itemize}
Following the same analysis in Theorem \ref{thm_redefine_special2}, we can show that the only possible troubled cell are located at the boundaries of the influence region and determine the possible types of the troubled cells, hence we skip the details. Moreover, since we did not modify the boundary cells, the partition lines keep the same at the two boundaries of the influence region.
\end{proof}

\subsubsection{From Theorem \ref{thm_redefine} to prove Theorem \ref{main}}
With all the preparation given above, we can prove Theorem \ref{main}. Actually, by Theorem \ref{thm_redefine} and Corollary \ref{time_useful}, we can redefine the numerical approximations within the influence region such that the only possible troubled cells in the influence regions are those at the boundaries, if $\Delta t<\lambda\Delta x$ with $\lambda=\frac4{a-b}$, and the troubled cells are mainly due to the strong shocks at the boundaries of the influence region. Such troubled cell or its neighbor, which were not updated, will be defined as an ETC, and the influence region of the new ETC will overlap with tht of the originial one, which contradicts the assuption that the ETCs are isolated. Therefore, no troubled cells exist in whole computational domain. Moreover, by Lemma \ref{modify_Harten}, the numerical scheme is TVD and MPP. Finally, we will perform the $L^2$ projection which keeps the maximum-principle and such procedure will not increase the total variation, and keeps the physical bounds for first-order schemes, hence we finish the proof.

\subsection{The stability analysis if the ETCs are not isolated}
In Theorem \ref{main}, we assume the ETCs are isolated. In practice, if there is only one shock in the exact solution, this assumption may be reasonable. However, if two shocks are interacting, then the ETCs may not be isolated. In this case, we may merge more cells. Based on the construction of the ETCs, we consider the cells to the right of the influence region only. Suppose the 5 cells, with the ETC as the center, in the influence region and the 5 cells on the right are given as $r_\ell,r_1,r_2,r_3,r_r,s_1,s_2,s_3,s_4,s_5$ from left to right, where $r's$ are the updated numerical approximations in the influence region given in Theorem \ref{thm_redefine}. For simplicity, we do not distinguish the troubled cells and the numerical approximations in the cells, e.g. we say $s_1$ is a troubled cell if $s_1$ is the numerical approximation in the troubled cell. 
The procedure is given as follows:
\begin{enumerate}
\item The influence region contains four cells, $r_\ell$, $r_1$, $r_2$, $r_3$. We search the cells from $r_3$ to the right, and select the ETC following the procedure in Section 2.

\item The influence region contains five cells, $r_\ell$, $r_1$, $r_2$, $r_3$, $r_r$ or six cells biased to the left. We search the cells from $r_r$ to the right, and select the ETC following the procedure in Section 2.

\item The influence region contains six cells, $r_\ell$, $r_1$, $r_2$, $r_3$, $r_r$, $s_1$. We search the cells from $s_1$ to the right, and select the ETC following the procedure in Section 2.

\item If the influence regions of the two ETCs are overlapping, we will merge the two influence regions together.
\end{enumerate}
\begin{theorem}
We can modify the $s_i's$ to be merged given above such that no troubled cells exist in the new influence region, except possibly the one on the right boundary of the new influence region. Therefore, after the whole procedure, the numerical approximations obtained from the first-order EL FV scheme satisfy the maximum-principle and they are total variation diminishing.
\end{theorem}
\begin{proof}
We will show that the procedure of merging the new influence region into the original one and updating the numerical approximations based on the new influence region will not create any new troubled cells in the original influence region. 

Actually, following the procedure given in Theorem \ref{thm_redefine} and Theorem \ref{firstcase}, the only possible new troubled cell that can be created by the new influence region is the one on the left boundary and it must be associated with a strong shock at the left boundary. However, we do not update the numerical approximation next to the left boundary, hence such a troubled cell is not created by the procedure of the new ETC but exists before we select the new ETC. Without loss of generality, we assume $r_2$ is the ETC whose influence region does not overlap with any other ETCs on its left. Then the left boundary of the influence region of $r_2$ is a troubled cell. However, this violates the selection procedure of the troubled cells given in Section \ref{sec2}.

Finally, for periodic boundary conditions, it is possible that the last ETC may interact with the first ETC. By the strategies of merging overlapping influence regions, no troubled cells exist between the first and last ETCs in the computational domain. By the same analysis given above, the last ETC will not create any new troubled cells at the right boundary of its influence region. So we complete the proof.
\end{proof}
Therefore, based on the above analysis, we can merge the influence regions of all non-isolated ETCs together. 

\section{Numerical experiments}
\label{secexample}
\setcounter{equation}{0}
\subsection{One-dimensional test results}
\label{sec1d}
In this subsection, we present several numerical examples on one-dimensional Burger's equation
\begin{equation} 
\label{eq:burgers1d}
u_t+(\frac{u^2}{2})_x = 0.
\end{equation}
The time-stepping size $\Delta t$ is defined by
\begin{equation}
\label{eqdt}
    \Delta t = \frac{C}{\max\{u_0\} - \min\{u_0\} } \cdot \Delta x, 
\end{equation}
where $u_0$ is the initial conditions. When $0< C < 4$, the TVD stability is guaranteed from the main Theorem. 
The classical $CFL$ number is defined as
\begin{equation*}
    CFL = \frac{\Delta t}{\Delta x} \cdot \max\lvert f'(u)\rvert.
\end{equation*}
Below $N$ is the number of cells used, and $T$ is the final evolving time. 

\begin{exa}\rm (Accuracy test).
\label{exa:cfl_error}
We consider the Burgers' equation with the initial condition $u_0(x) = \sin(x)$ on the domain $x \in [0, 2\pi]$ with periodic boundary conditions. To ensure that the TVD property is preserved, it is noted that the CFL number must not exceed 2. This is due to the restriction on the time-stepping size $\Delta t$ specified in equation \eqref{eqdt}.

In our tests, we evaluate the method at $T=0.8$ (before the shock) and $T=1.3$ (after the shock), where the shock is located at $x=\pi$. Table \ref{tab:tab1} shows the results of the mesh refinement study. With a CFL number of 1.95 (or $C=3.9$ in equation \eqref{eqdt}), we compute the $L^1$ error of the solution on the entire domain (for the case before the shock) and on $x \in [0,\pi-0.1] \cup [\pi+0.1, 2\pi]$ (for the case after the shock). The results indicate that the proposed scheme has achieved the desired first-order convergence.
Figure \ref{fig:cfl_error} plots the errors versus CFL numbers. We set the total number of cells to be 200 and compute the $L^1$ errors for CFL numbers ranging from 0.05 to 8.25 (or equivalently, $C\in [0.1, 16.5]$). It is observed that the $L^1$ errors remain at an excellent level, even when the CFL number is relatively large. The decrease in error with the increase of CFL (for CFL less than 1) is likely due to the decrease in the number of time steps when the CFL is larger, leading to a reduced accumulation of global error.
Finally, Figure \ref{fig:tvd_sin} shows the total variations versus time when $C=3.9$. The proposed scheme is observed to exhibit the desired TVD property.

\begin{table}[htb]
	\label{tab:tab1}
	\caption{Exa.\ref{exa:cfl_error}. Mesh refinement study, $CFL=1.95$.}
	\begin{center}
		\begin{tabular}{c  c c  c c }
			\hline
			{}  &  \multicolumn{2}{c}{T=0.8}  & \multicolumn{2}{c}{T=1.3}\\ \hline 
		 	N  & $L^1$ error      &  order  & $L^1$ error &  order \\ \hline 
				$100$  &   4.91e-03    & -- &   3.49e-03    & --  \\
				$200$   &  2.76e-03   & 0.83 &  1.91e-03   & 0.87       \\
				$300$   &  1.97e-03  & 0.83  &   1.25e-03  & 1.06 \\
				$400$   &  1.63e-03  & 0.64 &   9.43e-04  & 0.97  \\
				\hline
			\end{tabular}
		\end{center}
	\end{table}
\end{exa}

\begin{figure}[htb]
\begin{minipage}{0.48\linewidth}
    \centering
    \includegraphics[width=\textwidth]{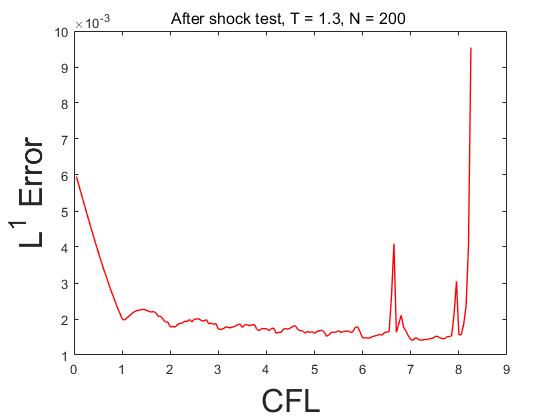}
    \caption{CFL vs. error plot for Example \ref{exa:cfl_error}. }
    \label{fig:cfl_error}
    \end{minipage}
\begin{minipage}{0.03\linewidth}
    \centering
    \ 
\end{minipage}
\begin{minipage}{0.48\linewidth}
    \centering
    \includegraphics[width=\textwidth]{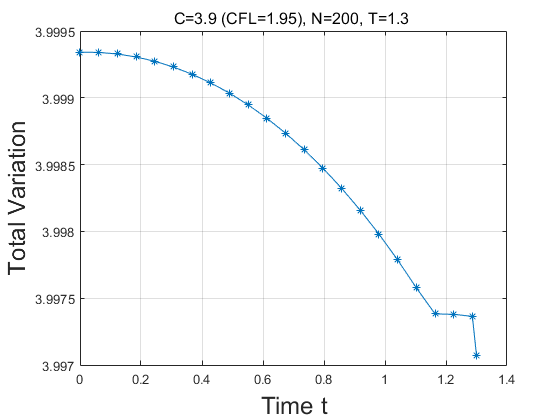}
    \caption{History of total variation for Example \ref{exa:cfl_error}.
    }
    \label{fig:tvd_sin}
    \end{minipage}
\end{figure}

\begin{exa}\rm (1D Riemann problems).
\label{exa:tvd}
We consider the Burgers' equation \eqref{eq:burgers1d} with Riemann initial condition
\begin{equation} \label{eq:shock}
       u_0(x) = \left\{ 
    \begin{aligned}
        2 &, &  x\leq 0,\\
        -1 &, & \text{otherwise,} 
    \end{aligned}
    \right.
\end{equation}
where $x \in [-\pi,\pi] $, and a constant boundary condition. The final time $T=3.6$ and numerical mesh is $N=100$. 
This test is to verify that the maximum time-stepping size in Theorem \ref{main} is sharp. 
Figure \ref{fig:tvd} shows the history of solutions' total variations under different choices of $C$. When $C=0.8$ ($CFL=\frac{8}{15}$), we observed that no cell-merging process occurred, and the total variations maintained constant. When $C$ equals $2.6$, $3.9$, and $4$ (or equivalently $CFL=\frac{26}{15}$, $2.6$, and $\frac{8}{3}$ ), the cell-merging process is triggered, yet the total variations remain constant. Finally, when $C$ exceeded 4, e.g. $C=4.9$, the total variations are observed to increase, i.e. the TVD property is violated. We present the TVD property of the solution with $C=3.9$ in Figure \ref{fig:shock1d}. 

\begin{figure}[htb]
\begin{minipage}{0.48\linewidth}
    \centering
    \includegraphics[width=\textwidth]{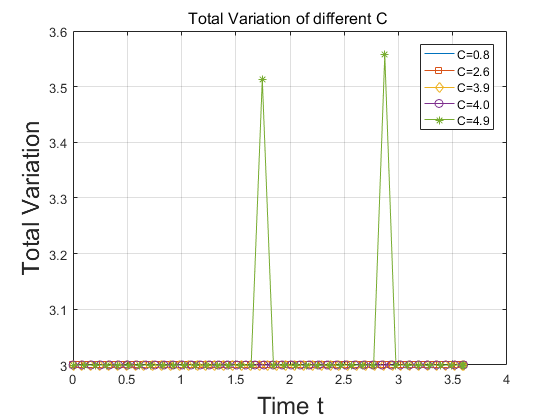}
    \caption{Example \ref{exa:tvd} with initial condition \eqref{eq:shock}. Total variation history for different $C$.}
    \label{fig:tvd}
    \end{minipage}
\begin{minipage}{0.03\linewidth}
    \centering
    \ 
\end{minipage}
\begin{minipage}{0.48\linewidth}
    \centering
    \includegraphics[width=\textwidth]{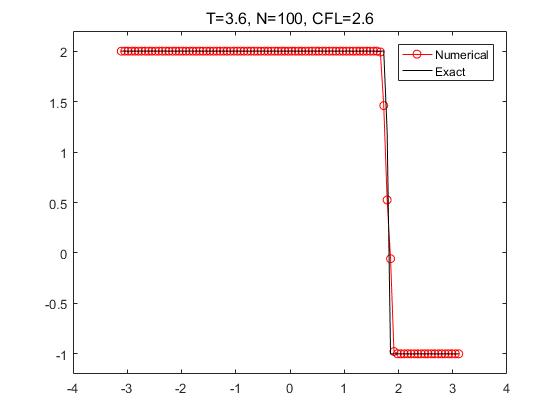}
    \caption{Example \ref{exa:tvd} with initial condition \eqref{eq:shock}. 
    }
    \label{fig:shock1d}
    \end{minipage}
\end{figure}
Next we present the results for the case of rarefaction wave with the initial condition
\begin{equation} \label{eq:raf}
       u_0(x) = \left\{ 
    \begin{aligned}
        -1 &, &  x\leq 0,\\
        1 &, & \text{otherwise,} 
    \end{aligned}
    \right.
\end{equation}
where $x \in [-\pi, \pi]$. We let $T=1.3$, $N=100$, and $CFL=1.95$ (or $C=3.9$). In fact, there is no time step limitation for the rarefaction wave problems since it does not need to merge the cells. Here we choose the $C=3.9$ to keep the consistency with the Theorem. The solution is plotted in Figure \ref{fig:raf1d}. An excellent resolution is observed even under a relatively large $\Delta t$.
Total variations of each time step are also shown in Figure \ref{fig:raf1dtv}. The total variation stays constant over time as expected.
\begin{figure}[htb]
\begin{minipage}{0.48\linewidth}
    \centering
    \includegraphics[width=\textwidth]{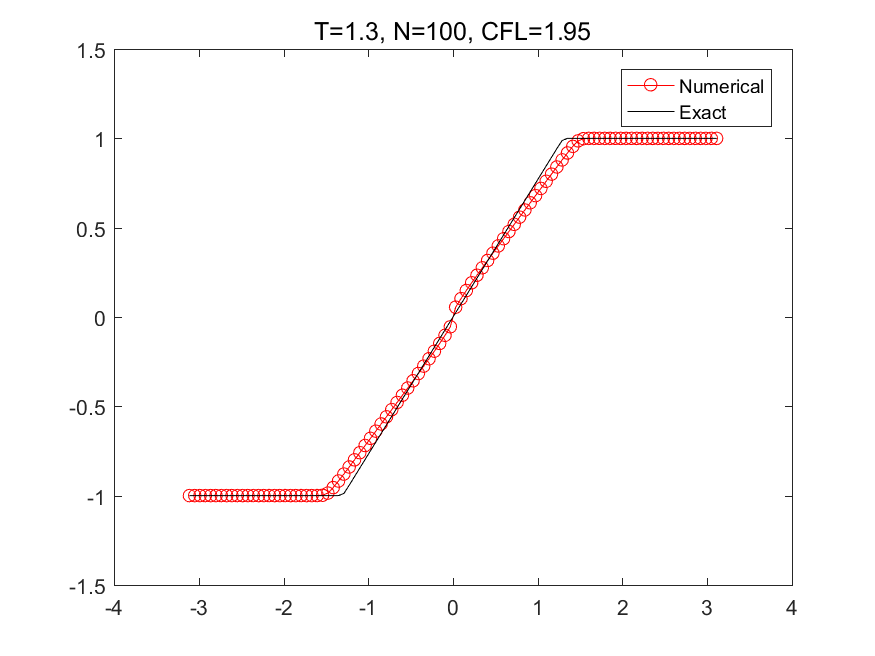}
    \caption{Example \ref{exa:tvd} with initial condition \eqref{eq:raf}. 
    }
    \label{fig:raf1d}
\end{minipage}
\begin{minipage}{0.03\linewidth}
    \centering
    \ 
\end{minipage}
\begin{minipage}{0.48\linewidth}
    \centering
    \includegraphics[width=\textwidth]{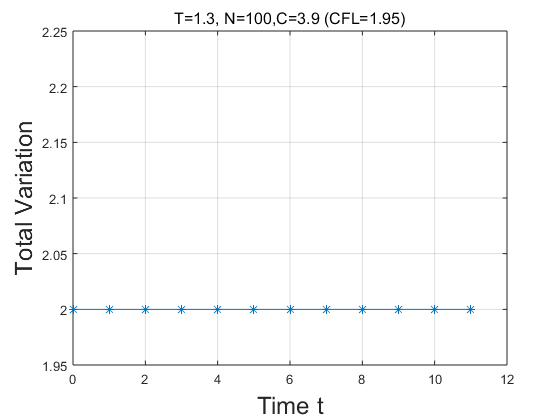}
    \caption{Example \ref{exa:tvd} with initial condition \eqref{eq:raf}. Total variation over time.
    }
    \label{fig:raf1dtv}
\end{minipage}
\end{figure}

Next, We present an extreme case to show the necessity of Cases I2 and I3 in Definition~\ref{ir}.
Consider the initial condition 
\begin{equation} \label{eq:extreme1}
       u_0(x) = \left\{ 
    \begin{aligned}
        2 &, &  x \leq 0, \\
        -0.6 &, &  0 < x\leq \Delta x,\\
        -2 &, & \text{otherwise,} 
    \end{aligned}
    \right.
\end{equation}
where $x\in [-\pi, \pi]$. We set final time $T=3$ and $N=100$. We test this initial condition with the merging procedure following case I4 in Definition \ref{ir}, that is \textbf{the merging procedure combines only five cells under any circumstances}. We choose $C=3.4$, $3.6$ and $3.9$ ($CFL = 1.7$, $1.8$, and $1.95 $) to compute the total variations at each time step. In Figure~\ref{fig:tvd2}, the initial total variation is $4$, and the total variation increases starting from the second step. As a result, the TVD property no longer holds, even though the constant $C$ is less than 4. To provide a comparison, we also test the same initial data again, but this time the merging procedure includes the Definition~\ref{ir}. In Figure~\ref{fig:tvd3}, we choose the same time-stepping size as in Figure~\ref{fig:tvd2}, i.e., $C=3.4$, $3.6$ and $3.9$ ($CFL = 1.7$, $1.8$, and $1.95 $). The TVD property is obtained in Figure~\ref{fig:tvd3}. Thus, the Definition~\ref{ir} is critical to obtain the optimal time-stepping size. 

\begin{figure}[htb]
\begin{minipage}{0.48\linewidth}
    \centering
    \includegraphics[width=\textwidth]{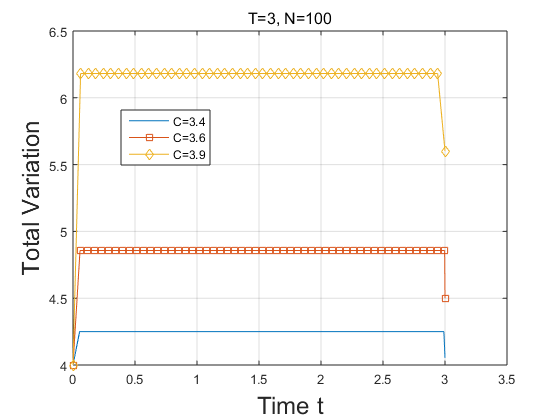}
    \caption{Exa.\ref{exa:tvd}. Initial condition of \eqref{eq:extreme1}. The numerical solution at $T=3$, $N=100$. Merging procedure excludes Def.~\ref{ir}.}
    \label{fig:tvd2}
\end{minipage}
\begin{minipage}{0.03\linewidth}
    \centering
    \ 
\end{minipage}
\begin{minipage}{0.48\linewidth}
    \centering
    \includegraphics[width=\textwidth]{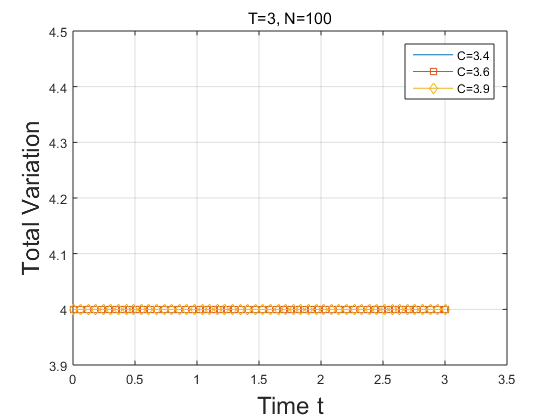}
    \caption{Exa.\ref{exa:tvd}. Initial condition of \eqref{eq:extreme1}. The numerical solution at $T=3$ and $N=100$. Merging procedure includes Def.~\ref{ir}.}
    \label{fig:tvd3}
\end{minipage}
\end{figure}

\end{exa}

\subsection{Two-dimensional test results}
\label{sec2d}
The proposed first order scheme can be extended to two-dimensional Burger's equation,
\begin{equation} 
\label{eq:burgers2d}
    u_t+(\frac{u^2}{2})_x + (\frac{u^2}{2})_y = 0,
\end{equation}
in a dimensional splitting fashion. The dimensional splitting method solves a 2D equation by alternating between solving two 1D equations, where we apply the first-order scheme to each one-dimensional equation via the Strang splitting. A detailed time splitting strategy under the EL FV framework can be found in \cite{Joseph2022}. All examples in the following use a constant boundary condition consistent with the Riemann initial condition. $N_x$ and $N_y$ are the cell numbers of $x, y$-direction respectively. 
The time-stepping size $\Delta t$ is bounded by,
\begin{equation*}
0< \Delta t < \frac{4 \cdot min\{\Delta x, \Delta y \} }{max\{u_0\} - min\{u_0\} }.  
\end{equation*}
The CFL is defined as
\begin{equation*}
   CFL =\Delta t \left({max_{x,y}\lvert f_x(u)\rvert/\Delta x + max_{x,y}\lvert f_y(u)\rvert/ \Delta y}\right),
\end{equation*}
and we adopt a 100$\times$100 grid.

\begin{exa}\rm (2D Burgers' equation with continuous initial condition).
\label{exa:2dsin}
We test 2D Burgers' equation \eqref{eq:burgers2d} with a continuous initial condition:
\begin{equation}
    \label{eq:conti2d}
          u_0(x) = \left\{ 
    \begin{aligned}
        (\sin(\pi x) \sin(\pi y) )^2  &, &  (x,y) \in (0,1)\times (0,1) \\
        0 &, & \text{otherwise,} 
    \end{aligned}
    \right. 
\end{equation}
where $(x,y) \in [0,2] \times [0,2]$. The initial condition is shown in Figure \ref{fig:sin0} in both mesh and contour plots. The $ CFL$ number is 7.6. The initial condition evolves up to $T=3$. In Figure \ref{fig:sin123}, we present both mesh and contour plots at $T=1$, $2$, and $3$. The solution generates a comet-like shape. The sharp discontinuity is formed at the head of the comet. The solution is observed to be stable, and the shock is captured correctly under a very large time stepping size.

\begin{figure}[htb]
    \centering
    \includegraphics[width=3in]{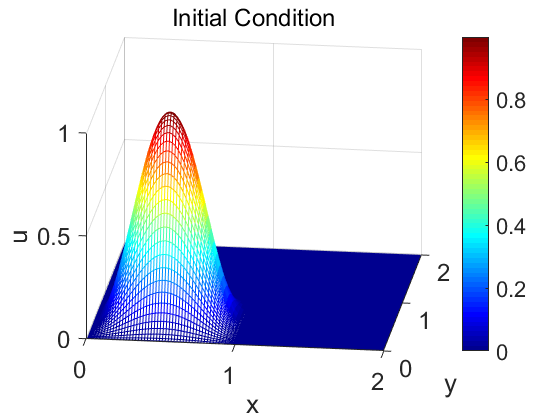}
    \includegraphics[width=3in]{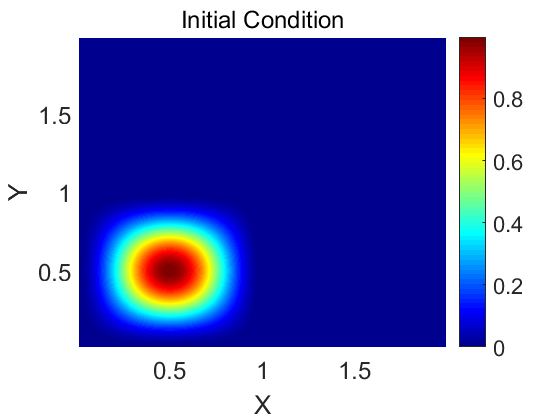}
    \caption{ Initial condition of Exa.\ref{exa:2dsin}.  $N_x=N_y=100$.}
    \label{fig:sin0}
\end{figure}

\begin{figure}[htb]
    \centering
    \includegraphics[width=2in]{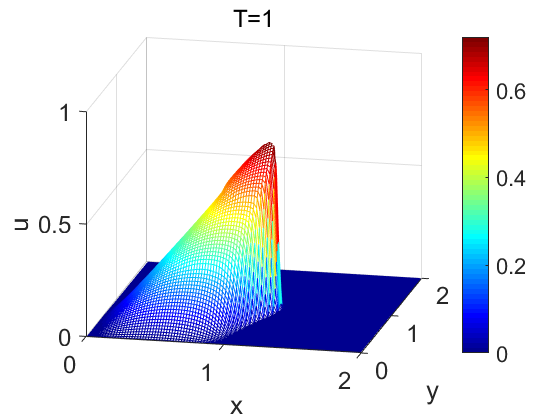}
    \includegraphics[width=2in]{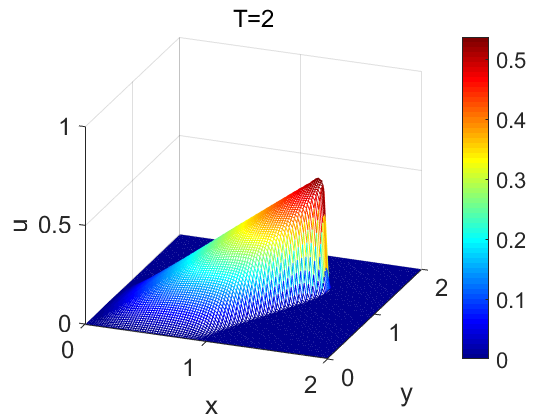}
    \includegraphics[width=2in]{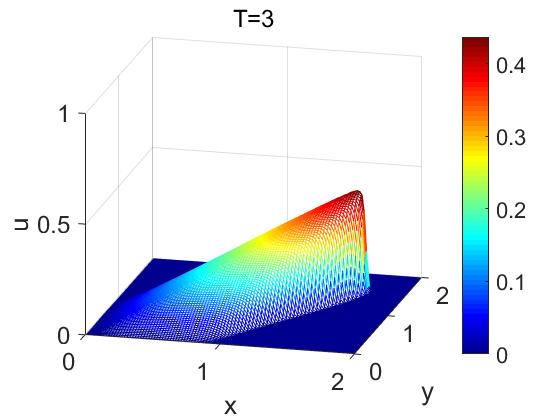} \\
    \includegraphics[width=2in]{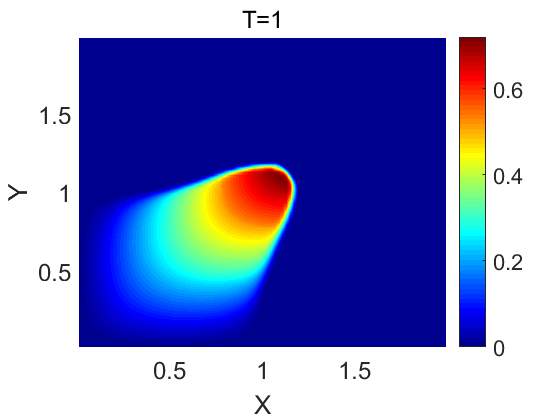}
     \includegraphics[width=2in]{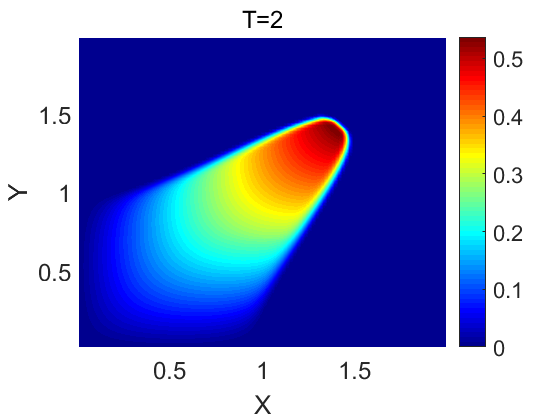}
      \includegraphics[width=2in]{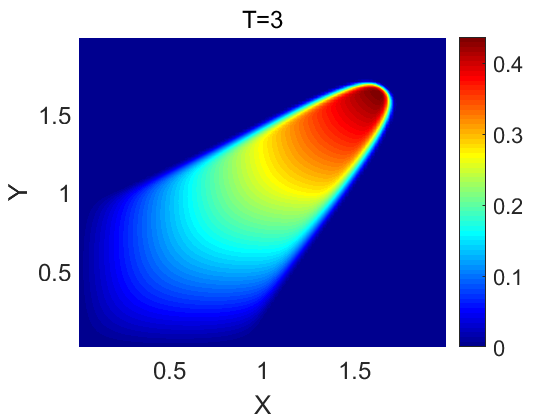}
    \caption{Exa.\ref{exa:2dsin}. $N_x=N_y=100$, $CFL=7.6$ .}
    \label{fig:sin123}
\end{figure}

\end{exa}

\begin{exa}\rm (2D Riemann problem).
\label{exa:dis2d}
Consider the 2D Burgers' equation with the Riemann Initial condition:
\begin{equation}\label{eq:rie2d}
          u_0(x) = \left\{ 
    \begin{aligned}
        1 &, &  (x,y) \in (0, 0.5]\times (0, 0.5], \\
        2 &, & (x,y) \in (-0.5, 0]\times [0, 0.5), \\
        3 &, & (x,y) \in [-0.5, 0)\times [-0.5, 0), \\
        4 &, & (x,y) \in (0,0.5) \times (-0.5, 0).
    \end{aligned}
    \right. 
\end{equation}
The domain is $(x,y) \in [-0.5,0.5] \times [-0.5,0.5]$. The analytic solution of this problem can be found in \cite{Yoon2007}.
The initial values (1,2,3,4) are assigned to each of the four quadrants, as shown in Figure \ref{fig:icshcok2d}. The $CFL$ number is $8.6$ and final time $T=0.1$. The solution contains three shocks and one rarefaction wave. One shock is formed at $\theta = 0$. Two shocks are formed at $\theta = \pi/2$ and $ \pi $, they interact to form a new shock. The rarefaction wave is formed at $\theta = 3 \pi/2$. The numerical scheme has the expected performance and sharply captures the shocks and rarefaction waves under a large time stepping size. 

\begin{figure}[htb]
    \centering
    \includegraphics[width=2.5in]{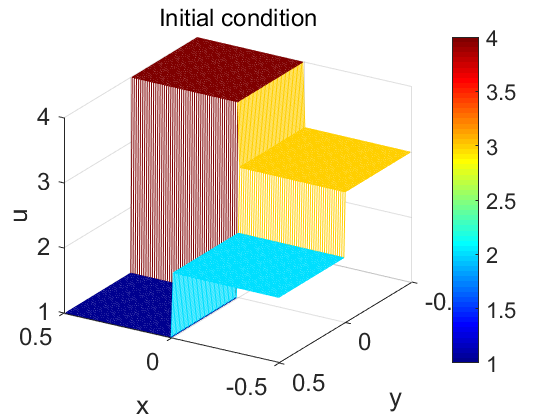}
    \includegraphics[width=2.5in]{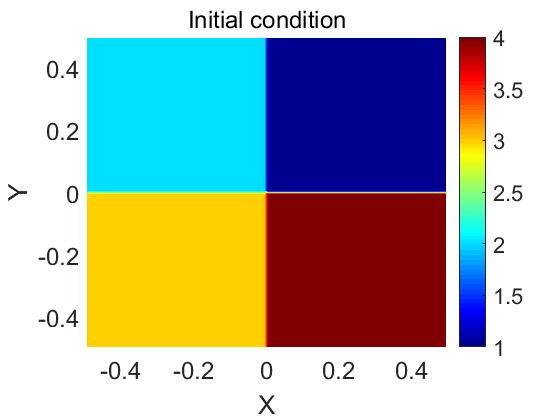}
    \caption{Initial condition of Exa.\ref{exa:dis2d}.  $N_x=N_y=100$. }
    \label{fig:icshcok2d}
\end{figure}

\begin{figure}[htb!]
    \centering
    \includegraphics[width=2.5in]{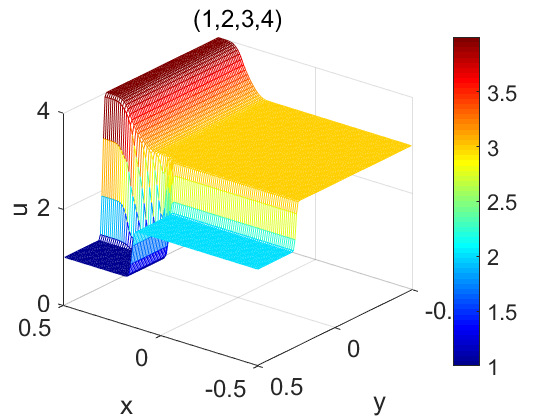}
    \includegraphics[width=2.5in]{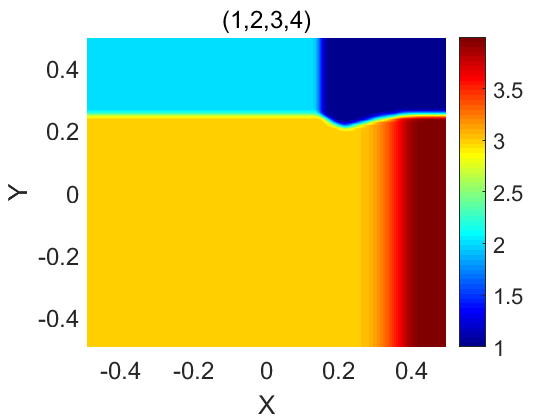}
    \caption{Exa.\ref{exa:dis2d},  $T = 0.1$, $N_x=N_y=100$, $CFL=8.6$. }
    \label{fig:shock2d}
\end{figure}

\end{exa}

\section{Concluding remarks}
\label{secend}
\setcounter{equation}{0}
In this paper, we constructed a first order EL FV scheme for nonlinear hyperbolic conservation laws. The main contribution was to introduce the merge strategy in the influence region of an ETC to handle shocks, leading to a scheme that is theoretically proved to be TVD and MPP. Moreover, the time step size is much larger than the traditional Eulerian FV schemes. We also theoretically demonstrated that if the size of the influence region was larger than some threshold, larger influence region would not result in larger time step sizes.
The extension of the first-order EL FV scheme to high-order ones can be done by applying high order spatial reconstruction with the minmod limiter and the bound-preserving limiter, and by applying high order SSP RK methods to temporal discretization in the space-time partitioned regions. Computational aspects of such extensions will be explored in \cite{Chen2022}.


\begin{thebibliography}{99}


\bibitem{Cockburn1989} B. Cockburn and C.-W. Shu, {\em TVB Runge-Kutta local projection discontinuous Galerkin finite element method for conservation laws II: general framework,} Mathematics of Computation, 52 (1989), 411-435.

\bibitem{Cai2021} X. Cai, J.-M. Qiu and Y. Yang, {\em An Eulerian-Lagrangian discontinuous Galerkin method for transport problems and its application to nonlinear dynamics}, Journal of Computational Physics, 439 (2021), 110392.

\bibitem{Chen2022} J. Chen, J. Nakao, J.-M. Qiu and Y. Yang, {\em A new Eulerian-Lagrangian finite volume Runge-Kutta WENO method for nonlinear hyperbolic problems}, in preparation.

\bibitem{gottlieb2009high} S. Gottlieb, D. Ketcheson and C-W. Shu, {\em High order strong stability preserving time discretizations}, Journal of Scientific Computing, 38 (2009), 251-289.


\bibitem{Harten} A. Harten, {\em High resolution schemes for hyperbolic conservation laws}, Journal of Computational Physics, 49 (1983), 357-393.

\bibitem{healy1998solution} R. Healy and T. Russell, {\em Solution of the advection-dispersion equation in two dimensions by a finite-volume Eulerian-Lagrangian localized adjoint method}, Advances in Water Resources, 21 (1998), 11-26.

\bibitem{huang2012eulerian} C-S. Huang, T. Arbogast and J. Qiu {\em An Eulerian--Lagrangian WENO finite volume scheme for advection problems}, Journal of Computational Physics, 231 (2012), 4028-4052.

\bibitem{leveque2002finite} L. Randall J,
{\em Finite volume methods for hyperbolic problems}, Cambridge university press (2022). 

\bibitem{Joseph2022} J. Nakao, J. Chen, and J.-M. Qiu, {\em An Eulerian-Lagrangian Runge-Kutta finite volume (EL-RK-FV) method for solving convection and convection-diffusion equations}, Journal of Computational Physics, 470 (2022), 111589.

\bibitem{qiu2008convergence} J.-M. Qiu, C.-W. Shu, {\em Convergence of Godunov-type schemes for scalar conservation laws under large time steps},
SIAM journal on numerical analysis,
46 (2008), {2211--2237}.

\bibitem{Yoon2007} D.K. Yoon and W.J. Hwang, {\em Two-dimensional Riemann Problem for Burgers' equation}, Bulletin of the Korean Mathematical Society, 45 (2008), 191-205.

\end{thebibliography}
\end{document}